\documentclass[english]{article}
\usepackage{graphicx}

\usepackage{amsthm}
\usepackage{amsmath}
\usepackage{amssymb}

\usepackage{graphicx}
\usepackage{subfig}
\usepackage{float}

\usepackage[utf8]{inputenc} % allow utf-8 input
\usepackage[T1]{fontenc}    % use 8-bit T1 fonts
\usepackage{hyperref}       % hyperlinks
\hypersetup{
   colorlinks=true,
   citecolor=black,%
   filecolor=black,%
   linkcolor=black,%
   urlcolor=blue
}

\usepackage{url}            % simple URL typesetting
\usepackage{booktabs}       % professional-quality tables
\usepackage{amsfonts}       % blackboard math symbols
\usepackage{nicefrac}       % compact symbols for 1/2, etc.
\usepackage{microtype}      % microtypography
\usepackage{mathrsfs}
\usepackage[usenames, dvipsnames]{color}

\usepackage[T1]{fontenc}
\usepackage{lmodern}

\usepackage{graphicx}
\usepackage{subfig}

\newtheorem{theorem}{Theorem}%[section]

\newtheorem{corollary}{Corollary}%[section]
\newtheorem{counter-example}{Counter Example}%[section]
\newtheorem{lemma}{Lemma}%[section]
\def \l {\left}
\def \r {\right}
\def \diag {\operatorname{diag}}
\def \dim {\operatorname{dim}}
\def \trace {\operatorname{tr}} 
\def \I {\operatorname{I}} % Identity matrix

\newcommand{\range}{\mathrm{range}}
\newcommand{\Stiefel}{\mathrm{St}}
\newcommand{\Grassman}{\mathrm{Gr}}
\newcommand{\Orth}{\mathrm{Orth}}
\newcommand{\GL}{\mathrm{GL}}
\newcommand{\Sym}{\mathrm{Sym}}

\newcommand{\comment}[1]{}

\newcommand{\RN}{\ensuremath{\mathbb{R}}}

\newcommand{\edit}[1]{{{#1}}}

\begin{document}

\title{PCA by Optimisation of Symmetric Functions\\ has no Spurious Local Optima}
\author{Armin Eftekhari and Raphael A.\ Hauser\thanks{Authors are ordered in the  alphabetical order. AE is with the Institute of Electrical Engineering at the \'{E}cole Polytechnique F\'{e}d\'{e}rale de Lausanne. RAH is with the Mathematical Institute at the University of Oxford and the Alan Turing Institute in London.  E-mails: \texttt{armin.eftekhari@epfl.ch} and \texttt{hauser@maths.ox.ac.uk}. The results in Section \ref{sec:pca by det} of this paper were previously presented at the International Conference on  Knowledge Discovery and Data 2018, but are included here as a special case to motivate and build intuition for the far more general results presented in Section \ref{sec:general conic}. The proof of the results in Section \ref{sec:pca by det} will also prepare the reader for the  ensuing more general  arguments.  }
}
\maketitle

\begin{abstract}

{Principal Component Analysis (PCA) finds the best linear representation of data, and is an indispensable tool in many learning and inference tasks. Classically, principal components of a dataset are interpreted as the directions that preserve most of its ``energy'', an interpretation that is theoretically underpinned by the celebrated Eckart-Young-Mirsky Theorem.}

This paper introduces many other  ways of performing PCA, with various geometric interpretations, and proves that the corresponding family of non-convex programs have no spurious local optima, \edit{while possessing only strict saddle points. These programs therefore loosely behave like convex problems and can be efficiently solved to global optimality, for example, with certain variants of the stochastic gradient descent.}

Beyond providing new geometric interpretations and enhancing our theoretical understanding of PCA, our findings might pave the way for entirely new approaches to structured dimensionality reduction, such as sparse PCA and nonnegative matrix factorisation. More specifically, we study an unconstrained formulation of PCA using determinant optimisation that might provide an elegant alternative to the  {deflating} scheme commonly used in sparse PCA. 
\end{abstract}

%%%%%%%%%%%%%%%%%%%%%%%%%%%%%%%%%%%%%%%%%%%%%%%%
\section{Introduction}\label{introduction}
%%%%%%%%%%%%%%%%%%%%%%%%%%%%%%%%%%%%%%%%%%%%%%%%

Let $A\in\mathbb{R}^{m\times n}$ be a data matrix, with  rows corresponding  to $m$ different data vectors, and columns corresponding to $n$ different features. Successful dimensionality reduction is at the heart of classification, regression, and other learning tasks that often suffer from the ``curse of dimensionality'', where having a small number of training samples in relation to the data dimension (namely, $m\ll n$) typically leads to overfitting~\cite{hastie2013elements}.

To reduce the dimension of data from $n$ to $p\le n$, consider \edit{a matrix} $X$ with orthonormal columns. 
Then the rows of $AX\in\mathbb{R}^{m\times p}$ correspond to the data vectors \edit{(namely, the rows of $A$)} projected onto  the column span  of $X$, \edit{which we denote by} $\range(X)$.  In particular, the new data matrix $AX$ has reduced dimension $p$, while the number $m$ of projected data vectors is unchanged, see Figure \ref{fig:Vis0}. 
 Principal Component Analysis~(PCA) is one of the oldest dimensionality 
reduction techniques that can be traced back to the work of Pearson \cite{pearson} and Hotelling \cite{hotelling}, motivated by the observation that often data lives near a lower-dimensional subspace of $\RN^n$, see Figure \ref{fig:Vis1}. PCA identifies this subspace by finding a suitable matrix $X$ that retains in $AX$ as much as possible of the energy of $A$, and 
 the optimal $X$ is called the \emph{loading matrix}. 
The columns of the loading matrix also reveal the hidden correlations between different features by identifying groups of variables that occur with jointly positive or jointly negative weights, for example in gene expression data \cite{alter4}.

\begin{center}
\begin{figure}[H]
\begin{center}
\includegraphics[width=0.8\textwidth]{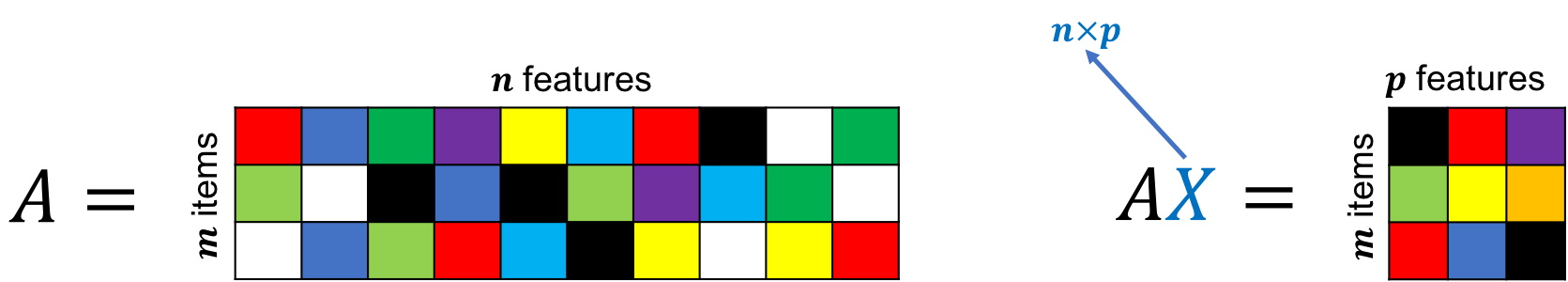}
\caption{This figure illustrates the simple and powerful concept of linear dimensionality reduction. The left panel shows a data matrix $A$, with rows corresponding to $m$ different data vectors and columns corresponding to $n$ different features. For a matrix $X\in\RN^{p\times r}$, the right panel shows the projected data matrix $AX$, containing  again $m$ data vectors (rows) but with only $p\le n$ features (columns). 
\label{fig:Vis0}  }
\end{center}
\end{figure}
\end{center}

\begin{center}
\begin{figure}[H]
\begin{center}
\includegraphics[width=0.4\textwidth]{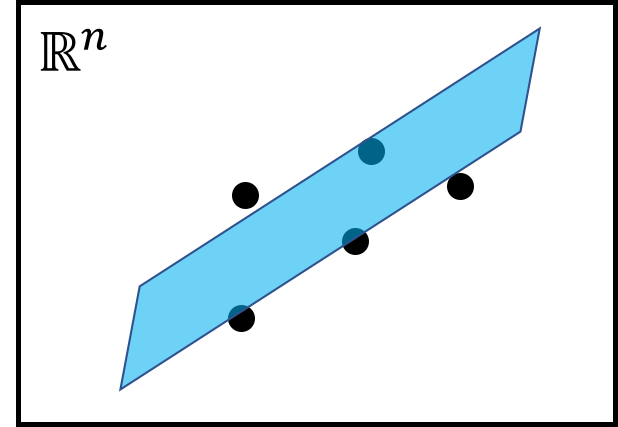}

\caption{
With each dot corresponding to a data vector with $n$ features, PCA finds a linear subspace (in blue) that best represents the data vectors by capturing  most of the energy of the dataset. 
\label{fig:Vis1}  }
\end{center}
\end{figure}
\end{center}

PCA is also the building block of other dimensionality reduction techniques such as sparse PCA \cite{journee2010generalized}, kernel PCA \cite{scholkopf2002learning,shawe2004kernel}, multi-dimensional scaling \cite{borg2013modern}, and \edit{Nonnegative Matrix Factorisation} (NMF) \cite{gillis2014and}. For  example, sparse PCA aims to find the important features of data by requiring the loading matrix to be sparse, namely, to have very few nonzero entries. Sparse PCA is useful for instance in studying gene expression data, where we are interested in singling out a small number of genes that are responsible for a certain trait or disease~\cite{alter2000singular}.  NMF, on the other hand, requires both $AX$ and $X$ to have nonnegative entries, which is valuable in recommender systems for instance  where the data matrix $A$ containing, say, film ratings is nonnegative and one would expect the same from the projected data matrix $AX$.

More formally, assume throughout this paper that  the data matrix $A\in\RN^{m\times n}$ is mean-centred. \edit{That is,} $\sum_{i=1}^m a_i = 0$, where $a_i\in\mathbb{R}^n$ is the $i$-th row of $A$, namely, the $i$-th data vector. For $p\le n$, let $\mathbb{R}^{n\times p}_p$ be the space of full-rank $n\times p$ matrices and consider the {\em trace inflation function} 
\begin{eqnarray}
f_{\trace}:\RN^{n\times p}_p&\rightarrow&\RN\nonumber\\
X&\mapsto&
\frac{\| AX\|_F^2}{\|X\|_F^2} = 
\frac{\trace(X^*A^*AX)}{\trace(X^*X)},\label{gtrace}
\end{eqnarray}
and the program
\begin{equation}\label{pca by trace}
\arg\max\left\{f_{\trace}(X):\;X\in\Stiefel(n,p)\right\}.
\end{equation}
Above, $\|\cdot\|_F$ and $\mbox{tr}(\cdot )$ return the Frobenius norm and  trace of a matrix, respectively, and $A^*$ is the transpose of matrix $A$. 
With $p\le n$,  $\Stiefel(n,p)$ above denotes the the Stiefel manifold,  the set of all $n\times p$ matrices with orthonormal columns.   Note that when $X\in \Stiefel(n,p)$, the denominator in the definition of $f_{\trace}(X)$ in \eqref{gtrace} is constant, but we include this term to highlight the structural similarity with other inflation functions that we will study later.

It is a consequence of the celebrated Eckart-Young-Mirsky (EYM) Theorem  that a Stiefel matrix $X\in\Stiefel(n,p)$ is a global maximiser of Program \eqref{pca by trace} if and only if it consists of  $p$ leading right singular vectors 
of $A$, namely, the right singular vectors of $A$ corresponding to its $p$ largest singular values \cite{eckart, mirsky}. In other words, Program~\eqref{pca by trace} performs PCA on the data matrix $A$:  The loading matrix, namely, global maximiser of Program~\eqref{pca by trace},  is a $p$-leading right singular factor $V_p\in\RN^{n\times p}$ of $A$, and the projected data matrix $AV_p\in\RN^{n\times p}$ contains the first $p$ principal components of  $A$. 

Note also that Program \eqref{pca by trace} is non-convex because $\Stiefel(n,p)\subset \RN^{n\times p}$ is a non-convex set. Even though non-convex, Program~\eqref{pca by trace} behaves like a convex problem in the sense that any local maximiser of Program~\eqref{pca by trace} is also a global maximiser. Indeed, it is also a consequence of the EYM Theorem   that Program~\eqref{pca by trace} does not have any spurious local maximisers. \edit{Moreover, all saddle points of this program are \emph{strict}, namely, have an ascent direction.}
%, if the spectral gap satisfies $\sigma_p(A)/\sigma_{p+1}(A)>1$. Here, $\sigma_p(A)$ is the $p$-th  largest singular value of $A$}. 
Therefore the non-convex  Program~\eqref{pca by trace} can be \edit{efficiently solved (to global optimality) using (certain variants of) the stochastic gradient descent, see for instance~\cite{jin2017escape,mokhtari2018escaping,sun2015nonconvex}.}  \edit{Even more efficiently, solving Program~\eqref{pca by trace} or equivalently} computing the loading matrix and the principal components  of $A$ can be done in $O( \max(m,n) p^2)$ operations using fast algorithms for Singular Value Decomposition~(SVD)\edit{, see for example Algorithm 8.6.1 in \cite{golub1996matrix}.} 

\paragraph{Motivation.}
Our motivation for this work was the following simple observation. The interpretation of PCA as a dimensionality reduction \edit{tool} suggests that it should suffice to  find a matrix 
$X\in\RN^{n\times p}_p$ whose columns span the optimal subspace, \edit{which corresponds to} $p$ leading right singular vectors of $A$. That is, one would expect $f_{\trace}$ in Program \eqref{pca by trace} to be a function on the Grassmannian $\Grassman(n,p)$, the set of all $p$-dimensional subspaces of $\RN^n$. In other words, one would like $f_{\trace}$ to be invariant under an arbitrary change of basis in its argument. 

That is of course not the case, \edit{as a quick inspection of \eqref{gtrace} reveals.}  Generally, \edit{we have} $f_{\trace}(X\Theta)=f_{\trace}(X)$, only when $\Theta\in\Orth(p) $, namely, when $\Theta\in\RN^{p\times p}$ \edit{itself} is an orthonormal matrix. Program \eqref{pca by trace} is thus inherently constrained to work with Stiefel matrices, a requirement that is not particularly onerous in the case of PCA but becomes a conceptual nuisance when considering structured dimensionality reduction, such as sparse PCA or NMF. Indeed, enforcing sparsity or nonnegativity in the  columns of $X$ in conjunction with orthogonality for the columns of $X$ tends to be very restrictive and is perhaps a questionable objective \edit{in the first place}. 

\paragraph{Contributions.} 
Motivated by the above observation, this paper introduces many other ways of performing PCA, with various geometric interpretations, and proves that the corresponding family of non-convex programs have no spurious local optima, \edit{while possessing only strict saddle points}. These new programs therefore \edit{loosely} behave like convex problems \edit{and can be solved to global optimality in polynomial time with, for example, the variants of stochastic gradient ascent in~\cite{jin2017escape,mokhtari2018escaping}.} More specifically, replacing $\trace$ in $f_{\trace}$ with any elementary symmetric polynomial yields an equivalent formulation for PCA, see the family of problems in \eqref{fq} and  the \edit{even} larger family of problems in \eqref{conic}.

Program \eqref{pca by trace} \edit{above} is indeed a member of this large family. Another notable member of this family is Program~\eqref{pca by det} below, \edit{which} is effectively unconstrained, and \edit{consequently} does \emph{not} require $X$ to have orthonormal columns. This observation is of particular importance in practice. As we show in Section \ref{structured}, this unconstrained formulation of PCA in Program \eqref{pca by det} \edit{potentially} allows for an elegant approach to structured PCA, in which we wish to impose additional structure on the loading matrix, such as sparsity or nonnegativity.

Let us add that it is known already that Program \eqref{pca by det} is equivalent to PCA \cite{horn1990matrix}, \edit{see \cite{national2002nist} for an application to {optimal design}  and \cite{hyvarinen2004independent} for an example in the context of {independent component analysis}}. 
\edit{Of course, this equivalence does not guarantee that Program~\eqref{pca by det}, like Program~\eqref{pca by trace}, can also be solved in polynomial time.} \edit{In this sense, our contribution} is that the non-convex Program \eqref{pca by det} has no spurious local optima, \edit{has only strict saddle points,} and  can therefore be solved efficiently by certain variants of the \edit{stochastic} gradient descent. Moreover, the introduction of the rest of this large family of equivalent formulations of PCA and their analysis in this work is \edit{the other novel aspect of this work}. 

\paragraph{Organisation.}
The rest of this paper is organised as follows. To present this work in an increasing order of complexity, we first introduce in Section \ref{sec:pca by det} the unconstrained formulation of PCA, namely, Program \eqref{pca by det}, and discuss in Section \ref{structured} its potential application in structured dimensionality reduction. In Section \ref{sec:general conic}, we then present Programs~(\ref{fq},\ref{conic}), a large family of equivalent formulations of PCA, of which both Programs~(\ref{pca by trace},\ref{pca by det}) are members. The claim that  all these programs are indeed equivalent to PCA and \edit{can be efficiently solved} is proven in Sections \ref{details}, \ref{sec:proof of general optimality}, and the appendices.  %\note{to update this when simulations are added.}

%%%%%%%
\section{PCA by Determinant Optimisation \label{sec:pca by det}}
%%%%%%%

In analogy to $f_{\trace}$ in \eqref{gtrace}, let us  define the {\em volume inflation function} by 
\begin{eqnarray}
f_{\det}:\RN^{n\times p}_p&\rightarrow&\RN\nonumber\\
X&\mapsto& \frac{\det(X^*A^*AX)}{\det(X^*X)},\label{gdet}
\end{eqnarray}
where $\det$ stands for determinant and, in analogy to Program \eqref{pca by trace}, consider  the program 
\begin{equation}\label{pca by det pre}
\arg\max\left\{f_{\det}(X):\;X\in\Stiefel(n,p)\right\}.
\end{equation}
Observe that Programs \eqref{pca by trace} and \eqref{pca by det pre}  coincide for $p=1$, namely, when we seek the leading principal component of the matrix $A$, in which case $X^*A^*AX$ and $X^*X$ are both positive scalars.  
Unlike  $f_{\trace}$,  note that $f_{\det}$ is  invariant under an arbitrary change of basis. Indeed, for arbitrary 
$X\in\RN^{n\times p}_p$ and $\Theta\in\GL(p)$, we have that  
\begin{align}
f_{\det}(X\Theta)& =  \frac{\det(\Theta)^2\det(X^*A^*AX)}{\det(\Theta)^2\det(X^*X)} \nonumber\\
& =f_{\det}(X),
\label{invariance of det}
\end{align}
where $\GL(p)$ is the general linear group, the set of all invertible $p\times p$ matrices. 
That is, $f_{\det}$ is naturally defined on the Grassmannian $\Grassman(n,p)$ and consequently Program \eqref{pca by det pre} is equivalent to the program 
\begin{equation}\label{pca by det}
\arg\max\left\{f_{\det}(X):\;X\in\RN^{n\times p}_p\right\}.
\end{equation}
Because $f_{\det}$ is invariant under any change of basis by \eqref{invariance of det}, Program \eqref{pca by det} inherently constitutes an optimization over the Grassmannian $\Grassman(n,p)$. 
Moreover, it is important that Program \eqref{pca by det} is effectively unconstrained because $\RN^{n\times p}_p$ is an open subset of $\RN^{n\times p}$ with nonempty  interior. To summarise, the drawback of Program~\eqref{pca by trace} \edit{in Section \ref{introduction} which served as the motivation of this work}  is overcome by Program \eqref{pca by det}, because it is an unconstrained optimisation program that involves an objective function defined naturally on the Grassmannian. 

A key observation of this paper is that Program \eqref{pca by det} appears to be a good model for dimensionality reduction. Indeed, note that $X^*A^*AX\in\RN^{p\times p}$ is the sample covariance \edit{matrix} of the projected data $AX$. Consider the normal distribution $\mathcal{N}(0,X^*A^*AX)$ with zero mean and covariance matrix $X^*A^*AX$, which has ellipsoidal level sets of the form 
\begin{equation}
\{z\in \mathbb{R}^p\,:\, z^* X^* A^* A X z = c\},
\label{eq:bounding box defined}
\end{equation}
for arbitrary $c\ge 0$. Let $B_c$ be the bounding box of this level set and note that the volume of $B_c$ is $c^p \sqrt{\det(X^*A^*AX)}$. We can therefore interpret Program \eqref{pca by det} as maximising  the volume of this bounding box. \edit{That is}, Program \eqref{pca by det} \edit{loosely-speaking} finds the directions that \edit{maximise the volume of the projected dataset.}
%preserve most of the volume of the dataset.  

In contrast, Program \eqref{pca by trace} maximises the energy of the projected data. That is, Program \eqref{pca by trace} maximises the diameter of the  above bounding box, namely,  $c \sqrt{\mbox{tr}(X^*A^*AX)}$, rather than its volume, see Figure \ref{fig:Vis2}. It is perhaps peculiar that $\mbox{tr}(X^*A^*AX)$ is commonly referred to as the ``total variance'' of the dataset, for this quantity does not play any role in the normalising constant of the normal distribution $\mathcal{N}(0,X^*A^*AX)$, whereas $\det(X^*A^*AX)$ does, in direct generalization of the role the variance plays in the one-dimensional case. 
\begin{center}
\begin{figure}[h]
\begin{center}
\includegraphics[width=0.7\textwidth]{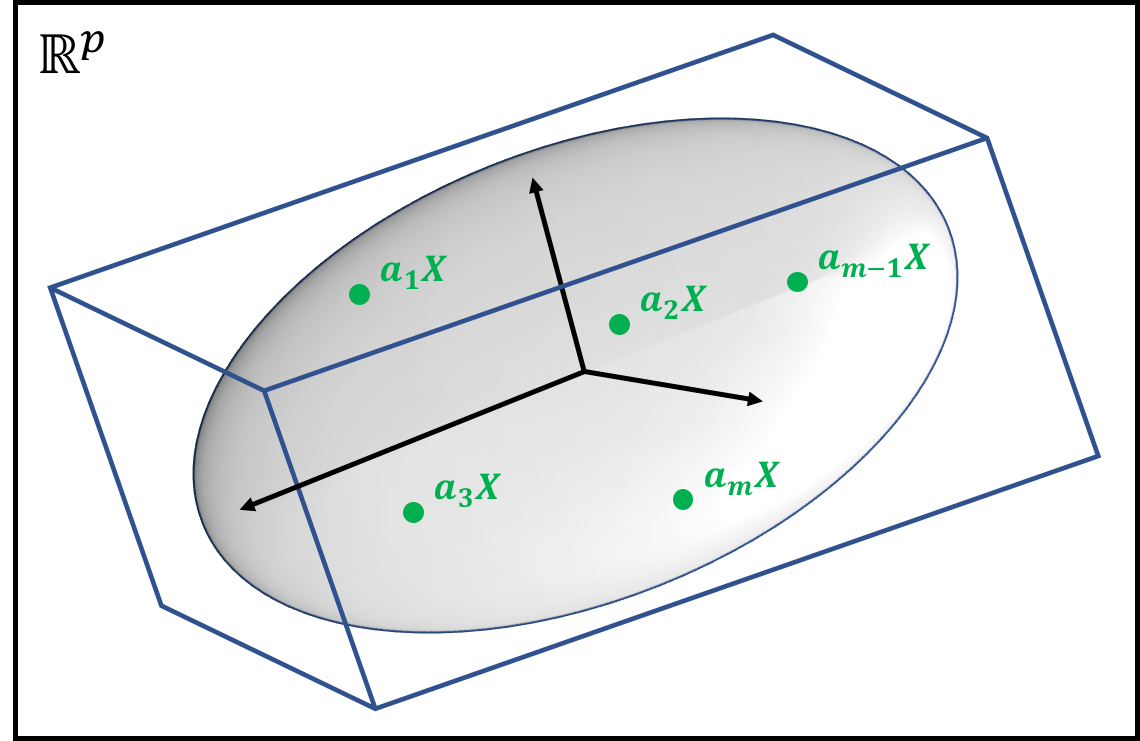}
\caption{
This figure illustrates the geometric intuition underlying this paper. Suppose that $a_1,\cdots,a_m\in \RN^n$ are the rows of the data matrix $A\in \RN^{m\times n}$, each representing a data vector. Then $a_1 X,\cdots ,a_m X \in \RN^p$ are the projected data vectors, with reduced dimension of $p\le n$. It is easy to see that the sample covariance matrix of these projected data vectors  is $X^*A^* AX\in \RN^{p\times p}$, with ellipsoidal level sets, one of which and its bounding box is displayed above.  Then Program \eqref{pca by trace} maximises the diameter of this bounding box, which is proportional to $\sqrt{\trace(X^* A^* AX)}$. In contrast, Program \eqref{pca by det} maximises the volume of this box, which is proportional to $\sqrt{\det(X^* A^* AX)}$. Remarkably, both of Programs (\ref{pca by trace},\ref{pca by det}) perform PCA of data matrix $A$, see Sections \ref{introduction} and \ref{sec:pca by det}. As discussed in Section \ref{structured}, Program \eqref{pca by det} is of particular importance in practice as it gives an elegant solution to the problem of structured linear dimensionality reduction. 
More generally, we also  show that maximising the sum of the volume squares of all $q$-dimensional facets of this bonding box is  equivalent to PCA of matrix $A$, for any $1\le q \le p$, see Section \ref{sec:general conic}. In particular, Programs (\ref{pca by trace}) and (\ref{pca by det}) are special cases with $q=1$ and $q=p$, respectively. 
\label{fig:Vis2}  }
\end{center}
\end{figure}
\end{center}
At any rate, we see that Programs (\ref{pca by trace},\ref{pca by det}) are both sensible approaches for linear dimensionality reduction, but that their geometric justifications are very different. Somewhat surprisingly, we find that  Program \eqref{pca by det} also performs PCA of $A$ and has no spurious local optima, exactly like Program \eqref{pca by trace}. The next result is proven in Section \ref{details}.

\begin{theorem}\label{det optimality} \textbf{\emph{(Determinant)}}
The following statements hold true: 
\begin{itemize}
\item[i) ] $\widetilde{X}\in\mathbb{R}^{n\times p}$ is a global maximiser of Program (\ref{pca by det}) if and only if there exists a $p$-leading right singular factor $V_p\in\mathbb{R}^{n\times p}$ of $A$ such that $\range(\widetilde{X})=\range(V_p)$. 
\item[ii) ] Program (\ref{pca by det}) does not have any spurious local optima, namely, any local maximum or minimum of Program  (\ref{pca by det}) is also a global maximum or minimum, respectively, and all other stationary points are  \edit{strict  saddle points. Moreover, if $\sigma_p(A) > \sigma_{p+1}(A)$,   at any such strict saddle point $X_s$, there exists an ascent direction $\Delta\in \mathbb{R}^{n\times p}$ such that 
\begin{equation}
\nabla^2 f_{\det}(X_s)[\Delta,\Delta]  \ge  f_{\det}(X_s) \l( \frac{\sigma_p^2(A)}{\sigma_{p+1}^2(A)} -1 \r) \| \Delta\|_F^2,
\label{eq:strict-genpca-thm}
\end{equation}
where the bilinear operator $\nabla^2 f_{\det}(X_s)$ is the Hessian of $f_{\det}$ at $X_s$. Above, $\sigma_p(A)$ is the $p$-th largest singular value of $A$.
}
\end{itemize}
\end{theorem}

\noindent In words, Part i) of Theorem \ref{det optimality} states that  Program \eqref{pca by det} performs PCA on the data matrix $A$, and therefore  Programs (\ref{pca by trace},\ref{pca by det}) are equivalent in this sense. Note that Program \eqref{pca by det} provides a different geometric interpretation of PCA based on maximising the ``volume'' of  projected data rather than its ``diameter'', which was the case in Program \eqref{pca by trace}. Even though we present  a new proof for the characterisation of the global maximisers of Program \eqref{pca by det} in Part i) of  Theorem \ref{det optimality}, this result can also be proved using interlacing properties of singular values, see Corollary 3.2 in 
\cite{horn1994topics}, or via the Cauchy-Binet formula \cite{karlin1988generalized}. 

The main contribution of Theorem \ref{det optimality} is  its Part ii) about the global landscape of the objective function $f_{\det}$,  stating that the non-convex Program \eqref{pca by det} behaves like a convex problem in the sense that any local maximiser (minimiser) of Program \eqref{pca by det} is also a global maximiser (minimiser). \edit{Moreover,   saddle points of Program~\eqref{pca by det} are strict.} 
In this way too, the two Programs (\ref{pca by trace},\ref{pca by det}) are similar, see Section \ref{introduction}.  Note that Part ii) of Theorem \ref{det optimality}
is crucial in the design of new dimensionality reduction algorithms: The instability of all stationary points except the global optima \edit{and the strictness of all saddle points establishes that, for example, {stochastic} gradient ascent, converges to the correct solution in polynomial time~\cite{jin2017escape,mokhtari2018escaping}}. That is, the non-convex Program~\eqref{pca by det} can be \edit{efficiently} solved to global optimality. However, as discussed in Section~\ref{introduction}, computationally efficient algorithms for PCA are already available and application of, say, stochastic gradient ascent to Program~\eqref{pca by det} is not intended to replace those algorithms. Instead, as discussed in Section~\ref{structured}, the unconstrained Program~\eqref{pca by det} \edit{potentially} opens up a radically new approach to structured PCA.

We remark that Theorem \ref{det optimality} is in line with a recent trend in computational sciences to understand the geometry and performance of non-convex programs and algorithms \cite{li,sun2015nonconvex,eftekhari2016snipe,burer2003nonlinear,boumal2016non,bhojanapalli2016dropping,bhojanapalli2016global,ge2016matrix,jin2016provable,soltanolkotabi2017theoretical,ge2017learning,ge2015escaping,ge2017optimization}.  While the available results do not apply to our problem, the underlying phenomena are closely related. Perhaps the closest result to our work is \cite{ge2016matrix}, stating that the (non-convex) matrix completion program has no spurious local optima when given access to randomly-observed matrix entries. This result in a sense extends the  EYM~Theorem~\cite{eckart, mirsky} to partially-observed matrices. 

From a computational perspective,  we may consider the program 
\begin{equation}\label{pca by log det}
\arg\max\left\{\log(f_{\det}(X)):\;X\in\RN^{n\times p}_p\right\},
\end{equation}
which is equivalent to Program \eqref{pca by det} but has better numerical stability. As a numerical example, we generated generic $U,V\in\Orth(100)$ and random matrix $A\in\RN^{100\times 100}$ with SVD $A=U\Sigma V^*$. The singular values of $A$, namely the entries of the diagonal matrix $\Sigma\in\RN^{100\times 100}$, were selected according to the power law. To be specific, we took $\sigma_i = i^{-1}$ to generate Figure \ref{fig:power1} and $\sigma_i=i^{-2}$ to generate Figure \ref{fig:power2}, for every $i\in[100]=\{1,2,\cdots,100\}$. For $p=5$, we let $V_p\in\RN^{n\times p}$ denote the first $p$ columns of $V$ and, by Theorem \ref{det optimality}, the unique maximiser of Programs (\ref{pca by det},\ref{pca by log det}). (Note that $V_p$ is also the unique maximiser of Program \eqref{pca by trace} by the EYM Theorem.)  In order to find $V_p$, we then applied  gradient  ascent  to Program \eqref{pca by log det} with fixed step size of $\rho = 5$ and random initialisation, producing a sequence of estimates $\{X_l\}_l\subset\RN^{n\times p}$. We also recorded the error $\| X_l {X}_l^\dagger - V_pV_p^* \|$ in the $l$th iteration, namely the sine of the principal angle between $\range(X_l)$ and $\range(V_p)$, which is plotted in Figures \ref{fig:power1} and \ref{fig:power2}. As predicted by Theorem \ref{det optimality}, the error vanishes in both examples as the algorithm progresses. \edit{We also refer the interested reader to \cite{hauser2018pca} for a comparison between Programs (\ref{pca by trace},\ref{pca by log det}), as well as  LAPACK's implementation of Lanczos’ method~\cite{golub1996matrix} for performing PCA. While the numerical results presented in \cite{hauser2018pca} are encouraging, a more comprehensive study is required to investigate the competitiveness of Program~\eqref{pca by log det} for PCA, as an alternative to more mainstream approaches \cite{golub1996matrix}.}

\edit{It might also be helpful to highlight the following practical consideration.}
Let $\widetilde{X}\in\mathbb{R}^{n\times p}$ denote a maximiser of Programs \eqref{pca by det} or \eqref{pca by log det}. \edit{Given $\widetilde{X}$, a few extra steps are required to compute the complete SVD of $A$, which we now list.} Let $\widehat{X}\in \RN^{n\times p}$ be an orthonormal basis for $\widetilde{X}$, which  can be computed in $O(np^2)$ operations by SVD.  Then computing the  SVD of 
$A\widehat{X}=\widehat{U}\widehat{\Sigma}\widehat{V}^*$ can be performed in merely $O(m p^2)$ operations and yields  the diagonal coefficients 
of $\widehat{\Sigma}$ as the $p$ leading singular values of $A$, as well as  $\widehat{U}$ and $\widehat{X}\widehat{V}$ as the corresponding 
$p$ leading left and right singular vectors of $A$, respectively.

\begin{center}
\begin{figure}[h]
\begin{center}
\includegraphics[width=0.5\textwidth]{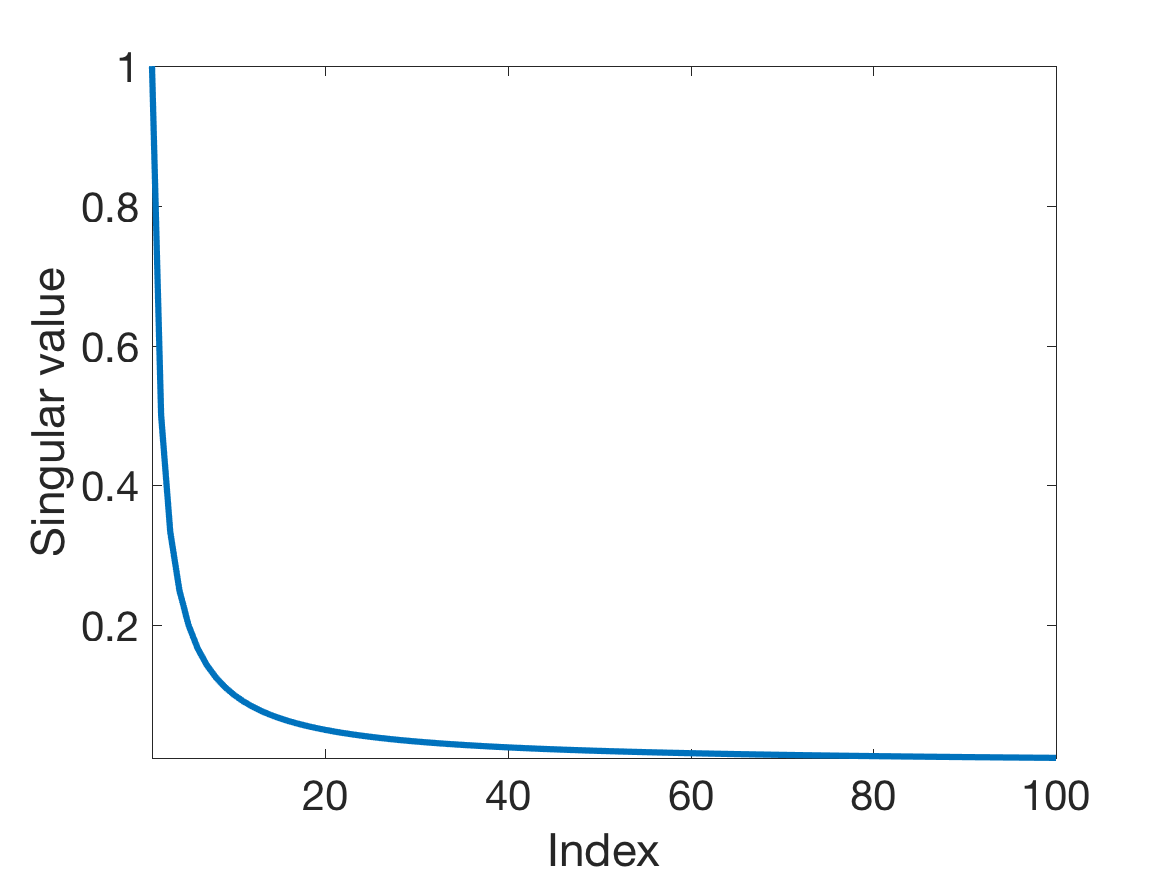}\includegraphics[width=0.5\textwidth]{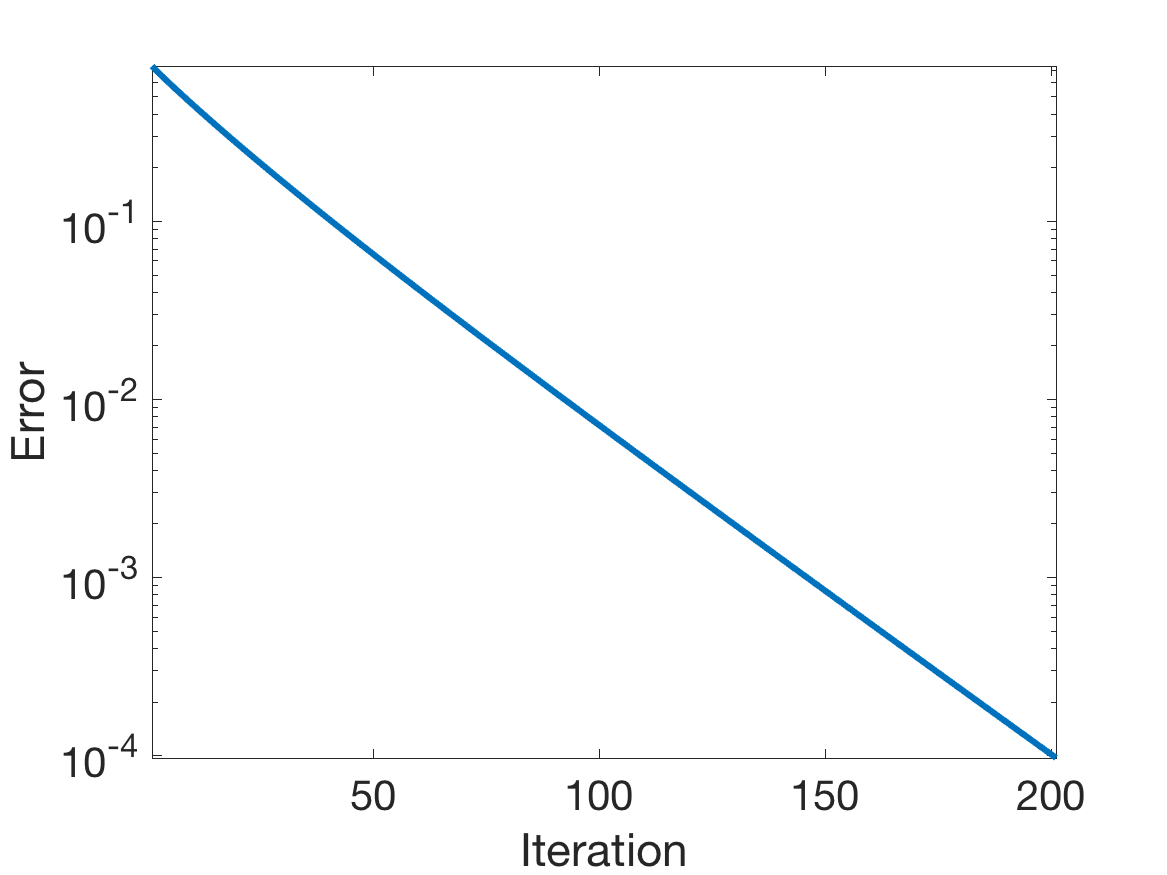}
\end{center}
\caption{The left panel shows the spectrum $\{\sigma_i\}_{i=1}^{100}$ of a randomly generated matrix $A\in\RN^{100\times 100}$ with $\sigma_i=i^{-1}$, and the right panel shows the progression of the gradient ascent algorithm with fixed step size, applied to Program \eqref{pca by log det}, see Section \ref{sec:pca by det} for details.
\label{fig:power1}  }
\end{figure}
\end{center}

\begin{center}
\begin{figure}[h]
\begin{center}
\includegraphics[width=0.5\textwidth]{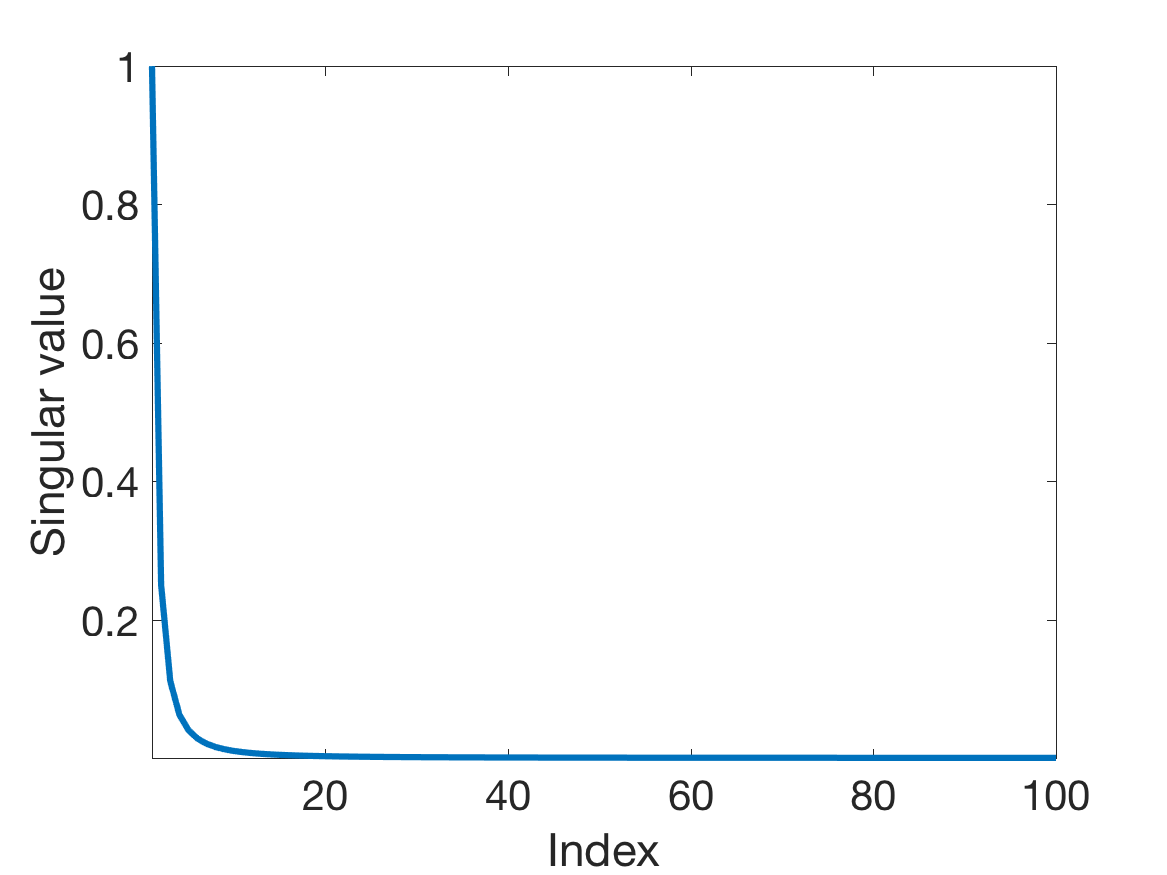}\includegraphics[width=0.5\textwidth]{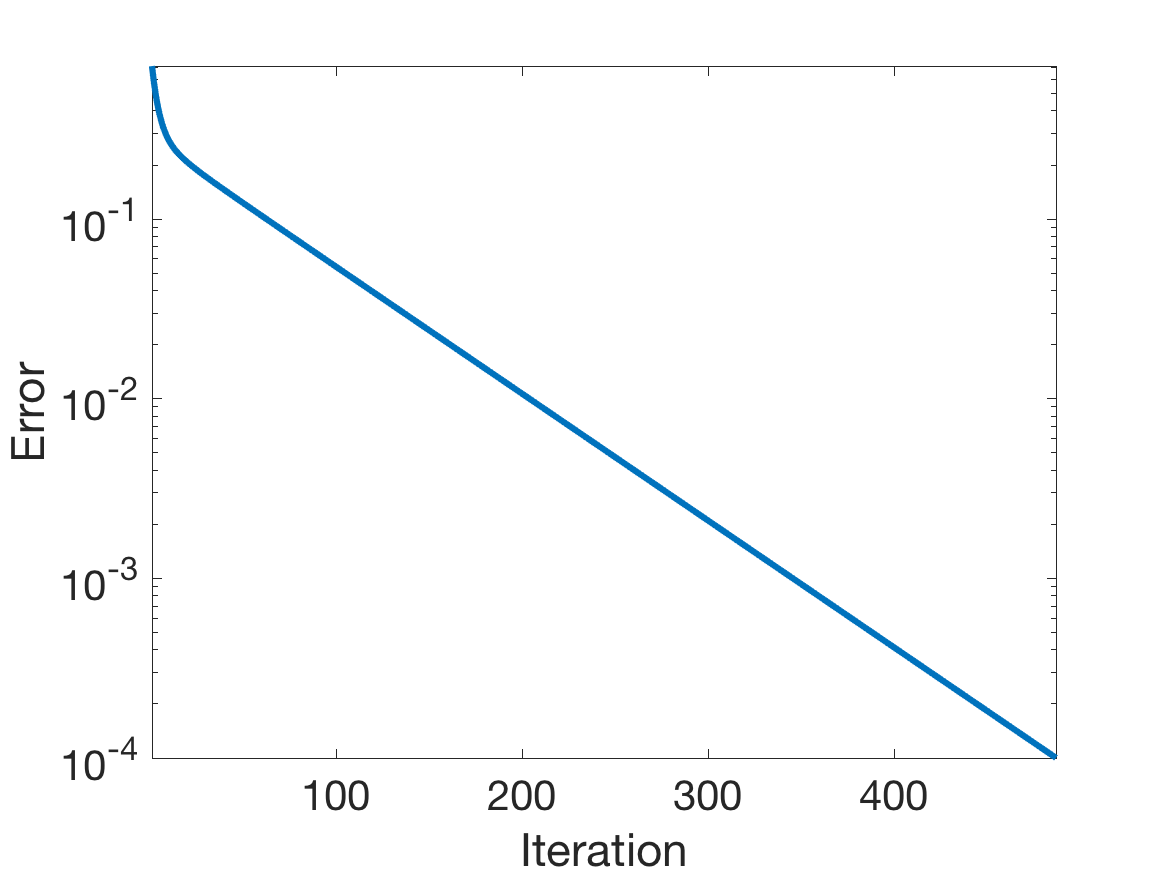}
\end{center}
\caption{ The left panel shows the spectrum $\{\sigma_i\}_{i=1}^{100}$ of a randomly generated matrix $A\in\RN^{100\times 100}$ with $\sigma_i=i^{-2}$, and the right panel shows the progression of the gradient ascent algorithm with fixed step size, applied to Program \eqref{pca by log det}, see Section \ref{sec:pca by det} for details. 
\label{fig:power2} }
\end{figure}
\end{center}

\section{Structured PCA \label{structured}}

\edit{As we will see  in this section,} the unconstrained formulation of PCA in Program~\eqref{pca by det} \edit{might be of particular interest in practice}, in contrast to Program~\eqref{pca by trace} which is restricted to the Stiefel manifold. \edit{Indeed,} the determinant formulation of PCA in Program \eqref{pca by det} \edit{might} allow for a more elegant approach to structured PCA, in which we wish to impose additional structure on the loading matrix, such as sparsity or nonnegativity. 

For the \edit{purposes} of this brief \edit{and informal} discussion, let us focus on sparse PCA, the problem of finding a small number of features that best describe the data matrix $A\in\RN^{m\times n}$. As one \edit{application}, when working with gene expression data, we are interested in a small number of features (genes) \edit{which} are responsible for certain traits or diseases \cite{johnstone2009consistency,deshpande2014information}.  The ``dual'' of sparse PCA can also be interpreted as data clustering.  

\edit{Loosely speaking}, sparse PCA is the problem of finding a \emph{sparse}\footnote{A sparse matrix has a small number of nonzero entries.} matrix $X\in\RN^{n\times p}$ that retains, in the projected data  $AX\in\RN^{m\times p}$, as much as possible of the energy of $A$. \edit{More formally}, sparse PCA might be formulated as a natural generalisation of Program \eqref{pca by trace}, namely, 
\begin{equation}
\arg\max\l\{ f_{\trace}(X):\; X\in \Stiefel(n,p) \mbox{ and } \|X\|_0 \le k\r\},
\label{eq:unnatural sparse pca}
\end{equation}
where $\|X\|_0$ is the number of nonzero entries of $X$, and the typically small integer $k$ is the \emph{sparsity level}. Note that Program \eqref{eq:unnatural sparse pca} forces $X$ to have orthonormal columns \emph{and} few nonzero entries, which tends to be restrictive and is also a somewhat questionable objective \edit{in the first place}. With a few exceptions, particularly  \cite{journee2010generalized}, this problem is often addressed by \emph{deflating} $A$, namely, finding the sparse principal components sequentially, that is,  one by one. Indeed, note that the Stiefel constraint from Program \eqref{eq:unnatural sparse pca} is redundant when $p=1$, namely, when $X$ is  a column vector. One could therefore find the leading sparse principal component of $A$, say $\widetilde{x}_1\in\RN^n$, by solving Program~\eqref{eq:unnatural sparse pca} with $p=1$,  remove its contribution from $A$ by forming $A_1 = A- A\widetilde{x}_1\widetilde{x}_1^*$, and then solve Program~\eqref{eq:unnatural sparse pca} with $A_1$ in place of $A$ to find the second sparse principal component $\widetilde{x}_2\in\RN^n$, and so on \cite{d2005direct}.  \edit{However,} deflating $A$ is believed to be inherently problematic when the problem is ill-posed~\cite{journee2010generalized}.

The determinant formulation of PCA in Program \eqref{pca by log det} might provide an elegant alternative to Program~\eqref{eq:unnatural sparse pca}. Recall that the feasible set of  Program \eqref{pca by log det} is an open subset of $\RN^{n\times p}$ with nonempty interior, and thus Program \eqref{pca by log det} is effectively unconstrained. We can therefore formulate sparse PCA by imposing a sparsity constraint on Program \eqref{pca by log det}, namely, 
\begin{equation}
\arg\max\l\{ \log(f_{\det}(X)):\; X\in \RN^{n\times p}_p \mbox{ and } \|X\|_0 \le k\r\},
\label{eq:natural sparse pca}
\end{equation}
which requires $X$ to be full-rank and sparse, relaxing the far more restrictive requirement of being Stiefel and sparse  in Program \eqref{pca by trace}. Note that removing the full-rank requirement in Program \eqref{eq:natural sparse pca} is impossible as that would mean $A$ has fewer than $p$ principal components and therefore the problem is ill-defined.

Similar ideas might be applied to nonnegative matrix factorisation, in which $X$ and $AX$ are both required to be nonnegative. More generally, the unconstrained nature of Program \eqref{pca by log det} might provide an entirely new approach to many structured dimensionality reduction problems, a \edit{research direction}  that remains to be explored in the future. \edit{In particular, what is the global geometry of Program~\eqref{eq:natural sparse pca}? What is the precise relationship between Programs~\eqref{eq:unnatural sparse pca} and \eqref{eq:natural sparse pca}?}
%\note{do we want  to add  simulations for sparse pca?}

\section{Generalisation to Positive Symmetric Polynomials \label{sec:general conic}}

So far, we have seen that maximising the trace objective function in Program~\eqref{pca by trace} and maximising the determinant objective function in Program~\eqref{pca by det} are  equivalent, and both provide the leading principal components of the data matrix $A$. Moreover, both non-convex programs can be solved to global optimality efficiently, \edit{see the discussion after Theorem~\ref{det optimality} for example.} Indeed, these claims \edit{for} Program \eqref{pca by trace} follow from the EYM Theorem~\cite{eckart,mirsky} and the claims \edit{for} Program \eqref{pca by det} follow from Theorem \ref{det optimality}, see Sections \ref{introduction} and~\ref{sec:pca by det}.

Note that both $\mbox{tr}(X^*A^*AX)$ and $\det(X^*A^*AX)$ are \emph{elementary symmetric polynomials},  namely, both are coefficients of the \emph{characteristic polynomial} of $X^*A^*AX$. More specifically, let $X^*A^*AX = W_X  \diag(\lambda_X)  W_X^*$ be the eigen-decomposition of $X^*A^*AX$, where $W_X\in\Orth(p)$ is an orthonormal matrix and the vector $\lambda_X\in\RN^p$ contains the eigenvalues of $X^*A^*AX$. Here, $\diag(\lambda_X)\in\RN^{p\times p}$ is the diagonal matrix formed by the vector $\lambda_X$. Then the characteristic polynomial associated with $X^*A^*AX$ takes $t\in\RN$ to 
\begin{align}
\det(\I_p + t X^*A^*AX) & = \det( W_X (\I_p+ t\diag(\lambda_X) ) W_X^*)
\qquad \l(W_X\in\Orth(p) \r) \nonumber\\
& = \det(W_X) \cdot \det(\I_p + t \diag(\lambda_X) )  \cdot \det(W_X^*)
\nonumber\\
& = \det \l[ 
\begin{array}{cccc}
1+t\cdot \lambda_{X,1} & 0 & \cdots & 0\\
0 & 1+t\cdot \lambda_{X,2} & \cdots & 0\\
\vdots & \vdots & \ddots & \\
0 & 0 & \cdots & 1+t \cdot \lambda_{X,p}
\end{array}
\r] 
\qquad \l( \det(W_X) = 1\r)
\nonumber\\
& = \prod_{i=1}^p (1+ t \cdot \lambda_{X,i}) \nonumber\\
& =: \sum_{q=0}^p s_q(X^*A^*AX) \cdot t^q,
\end{align} 
where the $q$th elementary symmetric polynomial $s_q:\Sym(p)\rightarrow \RN$ is the coefficient of $t^q$ above, namely,
\begin{eqnarray}
s_q:\Sym(p) &\rightarrow&\RN\nonumber\\
B &\mapsto& \sum_{1\le i_1< \cdots <i_q \le p}\,  \prod_{j=1}^q \lambda_{B,i_j},\label{sq}
\end{eqnarray}
with the convention that $s_0(B) = 1$. Above, $\{\lambda_{B,i}\}_{i=1}^p$ are the eigenvalues of $B\in\Sym(p)$, 
and $\Sym(p)$ denotes the set of symmetric $p\times p$ matrices over the real numbers. 
We also remark that elementary symmetric polynomials are \emph{spectral} functions in that they only depend on the eigenvalues of the input matrix.  As mentioned earlier, $\trace(X^*A^*AX)=s_1(X^*A^*AX)$ and $\det(X^*A^*AX) = s_p(X^*A^*AX)$. 

In analogy to trace and determinant objective functions (\ref{gtrace},\ref{gdet}), let us define 
\begin{eqnarray}
f_{s_q}:\RN^{n\times p}&\rightarrow&\RN\nonumber\\
X&\mapsto& \frac{s_q(X^*A^*AX)}{s_q(X^*X)},\label{fq}
\end{eqnarray}
for every $q\in [p]:=\{1,\dots,p\}$. In particular, $f_{s_1}=f_{\trace}$  in (\ref{gtrace}) and $f_{s_p}=f_{\det}$ in (\ref{gdet}) are two special cases. Lastly, in analogy to Programs (\ref{pca by trace},\ref{pca by det pre}), consider the program 
\begin{equation}
\arg\max\l\{ f_{s_q}(X) :\; X\in\Stiefel(n,p)\r\}.
\label{eq:general program}
\end{equation}
Again note that Programs (\ref{pca by trace}) and (\ref{pca by det pre}) are special cases of Program \eqref{eq:general program} for $q=1$ and $q=p$, respectively. Revisiting the geometric interpretation discussed in Section \ref{sec:pca by det}, we may also verify that ${f_{s_q}(X)}$ is proportional to the sum of volumes squared of all $q$-dimensional facets of the bounding box $B_c$, see right after \eqref{eq:bounding box defined} and also Figure \ref{fig:Vis2}. In this sense, Program \eqref{eq:general program} finds the projected data $AX$ that maximises this geometric attribute.

Generalising the EYM Theorem for Program \eqref{pca by trace} and Theorem \ref{det optimality} for Program \eqref{pca by det pre}, the following result states that Program \eqref{eq:general program} performs PCA of $A$ and has no spurious local optima for every $q\in[p]$, see Section~\ref{sec:proof of general optimality} for the proof. 
\begin{theorem}\label{general optimality} \textbf{\emph{(Elementary Symmetric Polynomials)}}
For every $q\in[p]$, the following statements hold true: 
\begin{itemize}
\item[i) ] $\widetilde{X}\in\mathbb{R}^{n\times p}$ is a global maximiser of Program (\ref{eq:general program}) if and only if there exists a $p$-leading right singular factor $V_p\in\mathbb{R}^{n\times p}$ of $A$ such that $\range(\widetilde{X})=\range(V_p)$. 
\item[ii) ] Program~(\ref{eq:general program}) 
does not have any spurious local optima, namely, any local maximum or minimum of Program~(\ref{eq:general program}) is also a global maximum respectively minimum, and all other stationary points are \edit{strict  saddle points.} 
%Moreover, if $\sigma_p(A) > \sigma_{p+1}(A)$,   at any such strict saddle point $X_s$, there exists an ascent direction $\Delta\in \mathbb{R}^{n\times p}$ such that 
%\begin{equation}
%\nabla^2 f_{s_q}(X_s)[\Delta,\Delta]  \ge  f_{s_q}(X_s) \l( \frac{\sigma_p^2(A)}{\sigma_{p+1}^2(A)} -1 \r) \| \Delta\|_F^2,
%\label{eq:strict-genpca-thm}
%\end{equation}
%where the bilinear operator $\nabla^2 f_{s_q}(X_s)$ is the Hessian of $f_{s_q}$ at $X_s$. Above, $\sigma_p(A)$ is the $p$-th largest singular value of $A$.}
\edit{
%Moreover, every saddle point of Program~(\ref{eq:general program}) is strict,  provided that $\sigma_p(A) > \sigma_{p+1}(A)$, where $\sigma_p(A)$ is the $p$-th largest singular value of $A$. 
%More specifically, there exists an ascent direction $\Delta\in \mathbb{R}^{n\times p}$ such that 
%\begin{equation}
%\nabla^2 f_{s_q}(X_s)[\Delta,\Delta]  \ge  f_{s_q}(X_s) \l( \frac{\sigma_p^2(A)}{\sigma_{p+1}^2(A)} -1 \r) \| \Delta\|_F^2,
%\label{eq:strict-detpca-thm}
%\end{equation}
%where the bilinear operator $\nabla^2 f_{s_q}(X_s)$ is the Hessian of $f_{s_q}$ at $X_s$.
 }
\end{itemize}
\end{theorem}
\noindent In words, Theorem \ref{general optimality} introduces a family of equivalent formulations for PCA, namely, Program \eqref{general optimality} for every $q\in [p]$. This family includes PCA by trace optimisation (Program \eqref{pca by trace}) and PCA by determinant optimisation (Program \eqref{pca by det}). 
In fact, maximising \emph{any} conic combination of elementary  symmetric polynomials also performs PCA. To be specific, for  nonnegative \edit{(but not all zero)} coefficients  $\{w_q\}_{q=0}^p$, consider the symmetric function 
\begin{eqnarray}
g_w:\Sym(p) &\rightarrow&\RN\nonumber\\
B &\mapsto& \sum_{q=0}^p w_q \cdot s_q(B),
\label{eq:psi-c}
\end{eqnarray}
where the elementary symmetric polynomial $s_q$ was defined in \eqref{sq}. Also define 
\begin{eqnarray}
f_{g_w}:\RN^{n\times p}&\rightarrow&\RN\nonumber\\
X&\mapsto& \frac{g_w(X^*A^*AX)}{g_w(X^*X)},\label{conic}
\end{eqnarray}
and consider the program 
\begin{equation}
\arg\max\l\{ f_{g_w}(X) :\; X\in\Stiefel(n,p)\r\}.
\label{eq:conic program}
\end{equation}
The following result is an immediate consequence of Theorem \ref{general optimality}, and states that Program \eqref{eq:conic program} for \emph{any} positive symmetric function  performs PCA, thus providing a broad class of equivalent formulations for PCA.   
\begin{corollary}\label{conic optimality} \textbf{\emph{(Positive Symmetric Functions)}}
Let $\{w_q\}_{q=0}^p$ be a set of nonnegative coefficients with at least one $w_q\neq 0$, and let $f_{g_w}$ 
be the function defined in (\ref{conic}). Then the following statements hold true: 
\begin{itemize}
\item[i) ] $\widetilde{X}\in\mathbb{R}^{n\times p}$ is a global maximiser of Program~(\ref{eq:conic program}) if and only if there exists a $p$-leading right singular factor $V_p\in\mathbb{R}^{n\times p}$ of $A$ such that $\range(\widetilde{X})=\range(V_p)$. 
\item[ii) ] Program~(\ref{eq:conic program}) does not have any spurious local maximisers, namely any local maximiser of Program~(\ref{eq:conic program}) is also a global maximiser, \edit{and all other stationary points are strict saddle points.} %\edit{Moreover, every saddle point of Program~(\ref{eq:conic program}) is strict, provided that $\sigma_p(A)>\sigma_{p+1}(A)$.}
\end{itemize}
\end{corollary}
\begin{proof}
Without loss of generality, we can assume that $w_0=0$. Indeed, since $s_0=1$ is constant by definition, setting $w_0=0$ does not change the optima and stationary points of Program \eqref{eq:conic program}. Note that $X^*X=\I_p$ for every  $X$ feasible to Programs (\ref{eq:general program},\ref{eq:conic program}). Therefore Program \eqref{eq:conic program} has the same optima and stationary points as 
\begin{equation}
\arg \max \l\{ g_w(X^* A^* AX):\; X\in \Stiefel(n,p) \r\},
\label{eq:aux program -1}
\end{equation} 
which, \edit{after recalling \eqref{eq:psi-c}, has in turn} the same optima and stationary points as 
\begin{equation}
\arg \max \l\{ \sum_{q=0}^p w_q  \cdot \frac{s_q(X^*A^*AX)}{s_q(X^*X)}\; X\in \Stiefel(n,p) \r\}.
\label{eq:aux program}
\end{equation} 
Recall  Theorem \ref{general optimality} about Program \eqref{eq:general program} for every $q\in[p]$. Because the coefficients $\{w_q\}_q$ are nonnegative by assumption, the claims in Theorem \ref{general optimality} extend to Program  \eqref{eq:aux program} and in turn to Program  \eqref{eq:aux program -1} and then to Program \eqref{eq:conic program}. This completes the proof of Corollary \ref{conic optimality}.
\end{proof}

\noindent In conclusion, this paper introduced a large family of equivalent interpretations of PCA which \edit{can all be solved to global optimality in polynomial time}. One member of this family is an unconstrained formulation of PCA that might lead in the future to developing new algorithms and techniques for structured PCA.

%%%%%
\section{Proof of Theorem \ref{det optimality}\label{details}}
%%%%%

We first begin with a change of variables. Let $A = U\Sigma V^*$ be the  SVD of $A$, where $U\in\mathbb{R}^{m\times m}$ and $V\in\mathbb{R}^{n\times n}$ are orthonormal matrices, and the diagonal matrix $\Sigma\in\mathbb{R}^{m\times n}$ is formed by the singular values of $A$ \edit{in nonincreasing order, denoted by} $\sigma_1 \ge \sigma_2 \ge \cdots$.  Let us set $\Gamma := \Sigma^*\Sigma \in\mathbb{R}^{n\times n}$ for short and note that 
\begin{equation}
\Gamma = 
\mbox{diag}(\gamma), 
\qquad \mbox{where } 
\gamma = \l[
\begin{array}{ccccc}
\sigma_1^2 & \cdots & \sigma_r^2 & 0 & \cdots
\end{array}
\r]^* \in\mathbb{R}^n,
\label{eq:defn-Gamma}
\end{equation}
where  $\mbox{diag}(\gamma)$  shapes the vector $\gamma$ into a diagonal matrix, and $r=\mbox{rank}(A)$ is the rank of $A$, namely, the number of positive singular values of $A$. Under the change of variables \edit{from} $X$ to $Y=V^*X$, Program (\ref{pca by det}) is  equivalent to 
\begin{equation}
\arg\max \l\{ \frac{\det(Y^* \Gamma Y) }{\det(Y^*Y)} \,:\, Y\in \mathbb{R}^{n\times p}_p\r\}. 
\label{det simplified}
\end{equation}
Without loss of generality, we will therefore assume that $A^*A=\Gamma$ in \eqref{gdet}, namely, we henceforth set
\begin{equation}
f_{\det}(X) = \frac{\det(X^* \Gamma X)}{\det(X^* X)}. \label{gdet new}
\end{equation} 
We will prove Theorem \ref{det optimality} by studying the stationary points of Program (\ref{pca by det}). This program is unconstrained and therefore a stationary point $X_s\in \mathbb{R}^{n\times p}_p$ of Program (\ref{pca by det}) is characterised by $\nabla f_{\det} (X_s)=0$. \edit{In light of the invariance in \eqref{invariance of det}, we can also  assume without loss of generality that $X_s \in \operatorname{St}(n,p)$. A stationary point $X_s \in \operatorname{St}(n,p)$ of Program \eqref{pca by det} thus satisfies }
\begin{equation}
\nabla f_{\det}(X_s) = 2\Gamma X_s \nabla \det(X_s^* \Gamma X_s) - 2\det(X_s^* \Gamma X_s) \cdot X_s\nabla \det(X_s^* X_s) = 0. 
\label{eq:st-cnd-first-proof}
\end{equation}
Note that \edit{the} determinant is a \emph{spectral function}, \edit{namely, it} only depends on the eigenvalues of its input matrix. More precisely, \edit{for a matrix $Z\in \mathbb{R}^{p\times p}$, it holds that}
\begin{equation}
\det(Z) = e_p(\lambda_Z) := \prod_{i=1}^p \lambda_{Z,i}, 
\label{eq:def of ep}
\end{equation} 
where $\lambda_{Z,i}$ is the $i$-th eigenvalue of $Z$ and the vector $\lambda_Z=[\lambda_{Z,1},\cdots,\lambda_{Z,p}]$ contains the eigenvalues of $Z$.  The derivative of a spectral function is well-known. To be specific, let $Z=W_Z \cdot \text{diag}(\lambda_Z) \cdot W_Z^*$ be the eigen-decomposition of $Z$, where $W_Z\in \Orth(p)$ is an orthonormal matrix and, as before, $\text{diag}(\lambda_Z)\in \RN^{p\times p}$ is the diagonal matrix formed by the vector $\lambda_Z$. For a spectral function $\phi:\Sym(p)\rightarrow \RN$, there exists a \emph{symmetric function} $\psi:\RN^p\rightarrow\RN$  such that 
\begin{equation}
\phi(Z) = \phi (\text{diag}(\lambda_Z)) = \psi(\lambda_z),
\label{eq:symm fcn}
\end{equation}
where we recall that a symmetric function \edit{is a function that} remains invariant  after changing the order of its arguments. We then have from \cite{lewis2002quadratic} that 
\begin{equation}
\nabla \phi (Z) = W_Z \cdot \mbox{diag}(\nabla \psi(\lambda_Z)) \cdot W_Z^*.
\label{eq:grad of spec}
\end{equation}
In our case, with $\phi=e_p$ and when $Z\in \GL(p)$, namely, when $Z$ is non-singular, we have that 
\begin{align}
\nabla \det(Z) & = W_Z \cdot \text{diag}(\nabla e_p(\lambda_Z)) \cdot W_Z^*
\qquad \text{(see \eqref{eq:grad of spec})}
\nonumber\\
& = W_Z \cdot \text{diag}\l( e_p(\lambda_Z) 
\l[
\begin{array}{ccc}
\frac{1}{\lambda_{Z,1}} & \cdots & \frac{1}{\lambda_{Z,p}}
\end{array}
\r]^* 
 \r) \cdot W_Z^*
 \qquad \text{(see \eqref{eq:def of ep})} \nonumber\\
 & =  \det(Z)  W_Z \cdot \text{diag}\l( 
\l[
\begin{array}{ccc}
\frac{1}{\lambda_{Z,1}} & \cdots & \frac{1}{\lambda_{Z,p}}
\end{array}
\r]^* 
 \r) \cdot W_Z^*.\qquad \text{(see \eqref{eq:def of ep})}
 \label{eq:grad-det-formula}
\end{align}
%which in particular yields that 
%\begin{equation}
%Z\in \GL(p) \Longrightarrow \nabla \det(Z) \in \GL(p). 
%\end{equation}
\edit{Returning to the proof and recalling that  $X_s \in \operatorname{St}(n,p)$,  \eqref{eq:grad-det-formula} allows us to write that 
\begin{align}
\nabla \det (X_s^* X_s ) = I_p. 
\label{eq:exp-nabla-det}
\end{align}
Likewise, after recalling that $X_s^*\Gamma X_s \in \GL(p)$ by assumption, it follows from \eqref{eq:grad-det-formula} that 
\begin{align}
\nabla \det(X_s^* \Gamma X_s ) = \det(X_s^* \Gamma X_s)\cdot (X_s^*\Gamma X_s)^{-1}. 
\label{eq:exp-nabla-det2}
\end{align}
Substituting (\ref{eq:exp-nabla-det},\ref{eq:exp-nabla-det2}), we find that \eqref{eq:st-cnd-first-proof} 
holds if and only if $\Gamma X_s = X_s X_s^*\Gamma X_s$, and since $X_s X_s^*$ is the projection into 
$\range(X_s)$, this is true if and only if $\range(\Gamma X_s)\subseteq\range(X_s)$. From this it follows that }
$X_s\in \operatorname{St}(n,p)$ \edit{such that} $X_s^*\Gamma X_s\in \GL(p)$ is a stationary point of Program~\eqref{pca by det}, if and only if 
\begin{equation}
\range(X_s)=\range(\Gamma X_s).
\label{eq:st cnd ranges}
\end{equation} 
%The above condition still holds after a change of basis from $X_s$ to $X_s\Theta$ for an invertible matrix $\Theta\in\operatorname{GL}(p)$.
%Without loss of generality, we can therefore assume that $X_s\in \operatorname{St}(n,p)$, namely,  we assume that $X_s$ has orthonormal columns. 
The following result, proved in Appendix \ref{sec:proof of lemma stationary}, characterises the stationary points of Program \eqref{pca by det}. That is, the next result characterises matrices $X_s\in \Stiefel(n,p)$ that satisfy \eqref{eq:st cnd ranges}. 
\begin{lemma}
\label{lem:stationary}For a singular value $\sigma_{j}$ of $A$, let $n_{j}$
denote the multiplicity of $\sigma_{j}$ \edit{in $A$}. Suppose that $X_s\in\Stiefel(n,p)$ satisfies $X_s^*\Gamma X_s \in \GL(p)$ and $\range( X_s) =  \range( \Gamma X_s)$. 
Then there
exists an index set $J\subset[n]$ such that 
\begin{enumerate}
\item $\{\sigma_{i}\}_{i\in J}$
are distinct and positive, and
\item for every $i\notin J$, the rows of $X_{s}$ corresponding to $\sigma_{i}$ are zero, and 
\item for every $i\in J$, the rows of $X_{s}$ corresponding
to $\sigma_{i}$ span a $\text{dim}(\sigma_{i})$-dimensional subspace
of $\mathbb{R}^{p}$, and 
\item for every distinct pair $\{i,i'\}\subset J$, the corresponding subspaces are orthogonal, and finally
\item $\sum_{i\in J}\dim(\sigma_{i})=p$. 
\end{enumerate}
Above, \edit{for every $i\in J$}, we set 
\[
\dim\left(\sigma_{i}\right)=\max\left[p-\sum_{i'\ne i}n_{i'},1\right].
\] 
\end{lemma}
\noindent Based on the characterisation of stationary points in Lemma \ref{lem:stationary}, we next calculate the Hessian of $f_{\det}$ at a stationary point of Program \eqref{pca by det}, which will later help us determine the stability of these stationary points. The following result is \edit{in fact more general, see Appendix \ref{sec:proof of stability lemma} for the proof. }
\begin{lemma} \label{lem:hessian calculated}
Consider a differentiable spectral function $\phi:\Sym(p)\rightarrow \mathbb{R}$ and  the associated symmetric function $\psi:\RN^p\rightarrow\RN$, see (\ref{eq:symm fcn}). Consider also the function 
\begin{eqnarray}
f_{\phi}:\RN^{n\times p}_p&\rightarrow&\RN\nonumber\\
X&\mapsto& \frac{\phi(X^* \Gamma X)}{\phi (X^*X)}.\label{eq:def of g}
\end{eqnarray}
Consider lastly $X_s\in \Stiefel(n,p)$ such that $\range(X_s)=\range(\Gamma X_s)$ and the corresponding index set $J=\{i_1,i_2,\cdots \}\subseteq[n]$ in Lemma \ref{lem:stationary}. Then we can assume without loss of generality that \edit{$X_s$ is block-diagonal}. Under this assumption, it holds that 
\begin{equation}
X_s^*\Gamma X_s = \diag(\widetilde{\gamma}_1),
\label{eq:simp assumption lem statement}
\end{equation}
where
\begin{equation}
\widetilde{\gamma}_1 := \l[
\begin{array}{ccc} 
\overset{\dim(\sigma_{i_1})}{\overbrace{\sigma_{i_1}^2 \cdots \sigma_{i_1}^2}} 
& \overset{\dim(\sigma_{i_2})}{\overbrace{\sigma_{i_2}^2 \cdots \sigma_{i_2}^2}} 
& \cdots 
\end{array}
\r]^* \in \RN^p.
\label{eq:def of gamma tilde}
\end{equation}
Moreover, let $K\subset[n]$ denote the index set corresponding to the nonzero rows of $X_s$. Then it holds that 
\begin{equation}
\nabla^2 f_\phi(X_s)[\Delta,\Delta] =  \frac{1}{\phi(\I_p)}  \sum_{i\in K^C} \sum_{j=1}^p   \l( \sigma_i^2 \partial_j \psi(\widetilde{\gamma}_1) - f_\phi(X_s) \partial_1 \psi(1_p) \r)  \Delta_{i,j}^2,
\label{eq:hessian lemma statement}
\end{equation}
for \edit{every}  $\Delta\in\RN^{n\times p}$ that is zero on the rows indexed by $K$.
Above, the bilinear operator $\nabla^2 f_\phi(X_s):\RN^{n\times p}\times \RN^{n\times p}\rightarrow \RN$ is the Hessian of $f_\phi$ at $X_s$. Also, \edit{$K^C$ is the complement of the set $K$,}  $\partial_i \psi(\widetilde{\gamma}_1)$ is the $i$-th entry of the gradient vector $\nabla \psi(\widetilde{\gamma}_1)\in\RN^p$, and $1_p\in\RN^p$ is the vector of all ones. 
\end{lemma}
\noindent In particular, when $\phi=\det$ and\edit{, consequently, $\psi = e_p$ in Lemma \ref{lem:hessian calculated}}, let us simplify the expression for \edit{the} Hessian in \eqref{eq:hessian lemma statement} by noting that 
\begin{equation*}
\phi(\I_p) = \det(\I_p) = 1,
\end{equation*}
\begin{align*}
\partial_j \psi(\widetilde{\gamma}_1)  & = 
\partial_j e_p(\widetilde{\gamma}_1) \nonumber\\
& = \prod_{i\ne j} \widetilde{\gamma}_{1,i} \nonumber\\
& = \frac{\prod_{i=1}^p \widetilde{\gamma}_{1,i}}{\widetilde{\gamma}_{1,j}} \nonumber\\
& = \frac{\det(\diag(\widetilde{\gamma}_1
))}{\widetilde{\gamma}_{1,j}} \nonumber\\
%& = \frac{\det(X_s^* \Gamma X_s)}{\sigma_{i_j}^{2}}
%\qquad \text{(see (\ref{eq:simp assumption lem statement},\ref{eq:def of gamma tilde} ))} \nonumber\\
& = \frac{\det(X_s^* \Gamma X_s)}{\widetilde{\gamma}_{1,j} \det(X_s^* X_s)},
\qquad (\edit{\Gamma = \text{diac}(\widetilde{\gamma}_1) \text{ and }} X_s \in \Stiefel(n,p)) \nonumber\\
& = \frac{f_{\det}(X_s)}{\widetilde{\gamma}_{1,j}},
\qquad \text{(see \eqref{eq:def of g})}.
\end{align*}
\begin{equation}
\partial_1 \psi(1_p) = \partial_1 e_p(1_p) = 1.
\label{eq:value of g at Xs}
\end{equation}
\edit{For any $\Delta\in \mathbb{R}^{n\times p}$ that is zero on the rows indexed by $K$,} substituting \edit{the above} values \edit{back into} \eqref{eq:hessian lemma statement} \edit{in Lemma \ref{lem:hessian calculated} with $\phi = \det$} yields that 
\begin{align}
\nabla^2 f_{\det}(X_s)[\Delta,\Delta] & 
 =  f_{\det}(X_s)  \sum_{i\in K^C}  \sum_{j=1}^p \l( \frac{\sigma_i^2}{\widetilde{\gamma}_{1,j}} -1  \r)  \Delta_{i,j}^2 .
 \label{eq:hessian-at-x}
\end{align}
If $\{\sigma_i\}_{i\in J}$ are  \emph{not} \edit{the unique numbers in the} $p$ leading singular values of $A$, then there exists $\Delta\in\RN^{n\times p}$ such that $\nabla^2 f_{\det}(X_s)[\Delta,\Delta]> 0$, namely, $\Delta$ is an ascent direction at $X_s$.  \edit{Indeed, let $\sigma_{i_0}$ with $i_0\in [n]$ be one of the $p$ leading singular values of $A$, not listed in $\{\sigma_i\}_{i\in J}$. Then it holds that 
\begin{align}
\sigma_{i_0}^2 > \min_{i\in J} \sigma_i^2 = \min_{j\in [p]} \widetilde{\gamma}_{1,j} =: \widetilde{\gamma}_{1,j_0},
\qquad \text{(see \eqref{eq:def of gamma tilde})}
\label{eq:wo-sepc-gap-req}
\end{align}
and,\footnote{\edit{In \eqref{eq:wo-sepc-gap-req}, the index $j_0$ might not be uniquely defined.}} moreover, 
\begin{align}
\frac{\sigma_{i_0}^2}{\widetilde{\gamma}_{1,j_0}} \ge \frac{\sigma_{p}^2}{\sigma_{p+1}^2}.
\label{eq:low-bnd-on-ratio}
\end{align}
Let $\Delta\in \RN^{n\times p}$ be such that $\Delta_{i_0,j_0}$ is its only nonzero entry and note that this choice of $\Delta$ is indeed zero on the rows indexed by $K$.  With this choice of $\Delta$ in \eqref{eq:hessian-at-x}, we find that 
\begin{align}
\nabla^2 f_{\det}(X_s)[\Delta,\Delta] & = f_{\det}(X_s) \l( \frac{\sigma_{i_0}^2}{\widetilde{\gamma}_{1,j_0}}   -1  \r) \Delta_{i_0,j_0}^2 >0.
\qquad \text{(see \eqref{eq:wo-sepc-gap-req})}
\end{align}
That is, if $\{\sigma_i\}_{i\in J}$ are \emph{not} the unique numbers in the $p$ leading singular values of $A$, then there exists an ascent direction at $X_s$. Moreover, if there is a nontrivial spectral gap $\sigma_p > \sigma_{p+1}$, then it also holds that 
\begin{align}
\nabla^2 f_{\det}(X_s)[\Delta,\Delta]& = f_{\det}(X_s) \l( \frac{\sigma_{i_0}^2}{\widetilde{\gamma}_{1,j_0}}   -1  \r) \Delta_{i_0,j_0}^2 
\nonumber\\
& \ge f_{\det}(X_s) \l( \frac{\sigma_{p}^2}{\sigma_{p+1}^2}  -1  \r) \Delta_{i_0,j_0}^2
\qquad \text{(see \eqref{eq:low-bnd-on-ratio})} \nonumber\\
& = f_{\det}(X_s) \l( \frac{\sigma_{p}^2}{\sigma_{p+1}^2}  -1  \r) \|\Delta\|_F^2,
\label{eq:det-spec-gap-result}
\end{align}
where the last line uses the fact that $\Delta_{i_0,j_0}$ is the only nonzero entry of $\Delta$ above. Likewise, we can establish that if $\{\sigma_i\}_{i\in J}$ are \emph{not} the unique numbers in the $p$ trailing singular values of $A$, then there exists a descent direction at $X_s$. We conclude that if $\{\sigma_i\}_{i\in J}$ are neither the unique numbers in the $p$ leading nor the $p$ trailing singular values of $A$, then $X_s$ is a strict saddle point (because it has both an ascent and a descent direction). 
}
%Likewise, if $\{\sigma_i\}_{i\in J}$ are \emph{not} the unique numbers in the $p$ smallest singular values of $A$, then there exists a descent direction at $X_s$. When combined, these findings imply that 
%In fact, a more careful assessment of \eqref{eq:hessian-at-x} reveals that 
%\begin{align}
%\nabla^2 f_{\det}(X_s,\Delta,\Delta) & \ge f_{\det}(X_s) \l( \frac{\sigma_p^2}{\sigma_{p+1}^2} -1 \r) \sum_{i\in K^C} \sum_{j=1}^p \Delta[i,j]^2 \nonumber\\
%& = f_{\det}(X_s) \l( \frac{\sigma_p^2}{\sigma_{p+1}^2} -1 \r) \| \Delta\|_F^2,
%\end{align}
% where the second line above follows because, by assumption, $\Delta$ is zero on the rows indexed by $K$. That is, the strictness of this saddle point/local minimum is determined by the spectral gap of $A$. 
% Likewise, if  $\{\sigma_i\}_{i\in J}$ are \emph{not} $p$ trailing singular values of $A$, there exists a decent direction at $X_s$. We conclude that, if $\{\sigma_i\}_{i\in J}$ is \emph{not} $p$ leading or $p$ trailing singular values of $A$, then $X_s$ must be a (strict) saddle point of $X_s$. 

  \edit{On the other hand, if $\{\sigma_i\}_{i\in J}$ are the unique numbers in the $p$  leading singular values of $A$, then all corresponding stationary points take the same objective value $f_{\det}$, which must (globally) maximise the (continuous) objective $f_{\det}$ on the compact set $\operatorname{St}(n,p)$. That is, every such stationary point $X_s$ is in fact a global maximiser of Program~\eqref{pca by det}. Likewise, if $\{\sigma_i\}_{i\in J}$ are the unique numbers in the $p$  trailing singular values of $A$, then all corresponding stationary points are global minimizers.}
%    \eqref{eq:hessian-at-x} implies that $X_s$ is a \edit{local} maximiser of Program \eqref{pca by det}. \edit{By \eqref{eq:simp assumption lem statement}, all such local maximisers share the same objective value $f_{\det}(X_s)$. 
%A similar calculation shows that, if $\{\sigma_j\}_{j\in J}$ are  $p$ trailing singular values of $A$, then  $X_s$ is a global minimiser of Program \eqref{pca by det}. 
This completes the proof of Theorem~\ref{det optimality}.

\edit{The advantage of Lemma \ref{lem:hessian calculated} above is that it gives an explicit expression for the Hessian of $f_\phi$, which will be used to prove Theorem \ref{general optimality}.  For the sake of completeness, however, let us show how Lemma \ref{lem:hessian calculated} can be replaced with a simpler argument here, described next.} In light of Lemma \ref{lem:stationary} and, if necessary, after a change of basis in \eqref{det simplified}, we can without loss of generality assume that a stationary point $X_s$ of Program \eqref{pca by det} is of the form 
\begin{equation}
X_s = \l[
\begin{array}{ccc}
c_{i_1}  & \cdots & c_{i_p}
\end{array}
\r] \in \Stiefel(n,p),
\end{equation}
where $\sigma_{i_1}\ge \cdots \ge \sigma_{i_p}$, and $c_i\in\RN^n$ is the $i$-th canonical vector that takes one at index $i$ and  zero elsewhere. 
If $X_s$ does not correspond to $p$ leading singular values of $A$, then there exists $i_0< i_1$ such that \edit{$\sigma_{i_0}> \min_{i_j} \sigma_{i_j}$. To simplify the presentation below, let us assume that in fact $\sigma_{i_0}> \sigma_{i_1}$. } Now consider the trajectory $\theta\rightarrow X(\theta)$ specified as 
\begin{equation}
X(\theta) =  \l[
\begin{array}{cccc}
c_{i_0} & c_{i_1}  & \cdots & c_{i_p}
\end{array}
\r]
\l[
\begin{array}{ccc}
\cos \theta & \sin \theta &  \\
-\sin \theta  & \cos \theta & \\
& & \I_{p-1}
\end{array}
\r]
\l[
\begin{array}{c}
0 \\
\I_{p}
\end{array}
\r]
\in \Stiefel(n,p)
,
\end{equation}
where the empty blocks in the square matrix above are filled with zeros. It is easy to verify that 
\begin{equation}
\frac{d^2 f_{\det}}{d\theta^2}(0) = \frac{2(\sigma_{i_0}^2-\sigma_{i_1}^2)}{\sigma_{i_1}^2} >0.
\end{equation}
That is, there exists an ascent direction at any stationary point that does not correspond to $p$ leading singular values of $A$. Likewise, one can verify that there exists a descent direction at any stationary point that does not correspond to $p$ trailing singular values of $A$, and now the rest of the proof of Theorem \ref{det optimality} follows as before.

\section{Proof of Theorem \ref{general optimality} \label{sec:proof of general optimality}}

The proof strategy is similar to that of Theorem \ref{det optimality} but with  some technical subtleties.  Without loss of generality, we assume again that $A^*A=\Gamma$ in \eqref{fq}, namely, we assume henceforth that
\begin{equation}
f_{s_q}(X) = \frac{s_q(X^*\Gamma X)}{s_q(X^*X)}. 
\label{eq:fsq-simplified}
\end{equation}
As with Theorem \ref{det optimality}, we will prove Theorem \ref{general optimality} by studying the stationary points of Program \eqref{eq:general program}, \edit{which we rewrite in the equivalent form
\begin{align}
\max_{X\in \RN^{n\times p}} \min_{\Lambda\in \RN^{p\times p} } f_{s_q}(X) + \langle  X^*X-I_p , \Lambda\rangle.
\label{eq:general program-rewritten}
\end{align} 
Therefore,  $X_s\in\Stiefel(n,p)$ is a stationary point of Programs (\ref{eq:general program},\ref{eq:general program-rewritten}), if and only if there exists $\Lambda_s\in  \RN^{p\times p}$ such that 
\begin{align}
\nabla f_{s_q}(X_s) + X_s ( \Lambda_s + \Lambda_s^*) = 0,
\end{align}
namely, when $\nabla f_{s_q}(X_s)$ belongs to the normal space to the Stiefel manifold at $X_s$ \cite{edelman1998geometry}. Without loss of generality, let us assume that $\Lambda_s = \Lambda_s^*$, so that the above condition simplifies to
\begin{align}
\nabla f_{s_q}(X_s) + 2 X_s  \Lambda_s = 0.
\label{eq:st cnd in general}
\end{align} In particular, since $X_s\in \operatorname{St}(n,p)$, we can multiply both sides above by $X_s^*$ and solve for $\Lambda_s$ above to obtain that 
\begin{align}
\Lambda_s = - \frac{1}{2} X_s^* \nabla f_{s_q}(X_s).
\label{eq:exp-for-lambda-s-thm2}
\end{align}}%\begin{equation}
%\nabla f_{s_q}(X_s) \in \mathrm{N}_{X_s} \Stiefel(n,p),
%\end{equation}
%where  $\nabla f_{s_q}(X_s)$ is the gradient of $f_{s_q}$ at $X_s$ with respect to the standard inner product of $\RN^{n\times p}$, and $\mathrm{N}_{X_s} \Stiefel(n,p)$ is the normal space to the Stiefel manifold at $X_s$, namely,
%\begin{equation}
%\mathrm{N}_{X_s} \Stiefel(n,p) = \l\{ X_s C :\; C\in \Sym(p) \r\},
%\end{equation} 
%see \cite{edelman1998geometry}. 
%Therefore, $X_s\in \Stiefel(n,p)$ is a stationary point of Program \eqref{eq:general program} if and only if  there exists $C\in\Sym(p)$ such that 
%\begin{equation}
%\nabla f_{s_q}(X_s) = X_s C,
%\end{equation}
%or, equivalently, 
\edit{Next, from \eqref{eq:fsq-simplified}, it follows that 
\begin{equation}
\nabla f_{s_q} (X_s) = \frac{2\Gamma X_s \nabla s_q(X_s^* \Gamma X_s)}{s_q(X_s^* X_s)} - \frac{2s_q(X_s^* \Gamma X_s) \cdot X_s\nabla s_q(X_s^* X_s)}{s_q(X_s^* X_s)^2}.
\label{eq:st cnd in general updated}
\end{equation}}Let us examine the above expression more carefully.  
For $Z\in\RN^{p\times p}$ with eigen-decomposition $Z=U_Z\diag(\lambda_Z)U_Z^*$, note that the symmetric function corresponding to $\phi = s_q$ is 
\begin{equation}
\psi(\lambda_Z)= e_q(\lambda_Z) := \sum_{1\le i_1 < \cdots < i_q \le p} \,\prod_{j=1}^q \lambda_{Z,i_j}.
\label{eq:def of eq}
\end{equation}
For a nonsingular matrix $Z\in\GL(p)$, it is then not difficult to verify that 
\begin{align}
\nabla e_q(\lambda_Z) & = \l[
\begin{array}{ccc}
e_{q-1}(\lambda_Z^1) & \cdots & e_{q-1}(\lambda_Z^p)
\end{array}
\r]^*,
\label{eq:grad of symm fcn general}
\end{align}
where $\lambda_Z^i\in\RN^{p-1}$ is formed from $\lambda_Z\in\RN^p$ by removing its $i$-th entry, namely $\lambda_{Z,i}$. Using \eqref{eq:grad of spec}, we immediately find that 
\begin{align}
\nabla s_q(Z)= W_Z
\diag\l(
 \l[
\begin{array}{ccc}
e_{q-1}(\lambda_Z^1) & \cdots & e_{q-1}(\lambda_Z^p)
\end{array}
\r]^* \r)
W_Z^* . 
\label{eq:exp-grad-general}
\end{align}
\edit{Recalling that $X_s\in \operatorname{St}(n,p)$ and using \eqref{eq:exp-grad-general}, we calculate the gradients involved in \eqref{eq:st cnd in general updated} as 
%\begin{align}
%s_q(X_s^* X_s ) & = s_q(I_p) \qquad (X_s\in \operatorname{St}(n,p)) \nonumber\\
%& = {p \choose q}, \qquad \text{(see \eqref{sq})}
%\end{align}
%\begin{align}
%s_q(X_s^* \Gamma X_s) & = 
%\end{align}
%\begin{align}
%\nabla s_q(X_s^*X_s) & = 
%\end{align}
\begin{align}
\nabla s_q(X_s^* X_s)  & = { p-1 \choose q-1} X_s^* X_s 
\qquad \text{(see (\ref{eq:def of eq},\ref{eq:exp-grad-general}))}  \nonumber\\
& = { p-1 \choose q-1} I_p,
\qquad (X_s \in \operatorname{St}(n,p))
\label{eq:nabla-general-1}
\end{align}
\begin{align}
\nabla s_q(X_s^* \Gamma X_s) & = X_s^* \diag\l( 
\left [
\begin{array}{ccc}
e_{q-1}(\widetilde{\gamma}_1^1) & \cdots & e_{q-1}(\widetilde{\gamma}_1^p)
\end{array}
\right]^*
 \r) X_s
 \qquad \text{(see (\ref{eq:simp assumption lem statement},\ref{eq:exp-grad-general}))}  \nonumber\\
 & = \diag\l( 
\left [
\begin{array}{ccc}
e_{q-1}(\widetilde{\gamma}_1^1) & \cdots & e_{q-1}(\widetilde{\gamma}_1^p)
\end{array}
\right]^*
 \r)
 \qquad \text{(see Lemma \ref{lem:hessian calculated})} \nonumber\\
 & =: \diag(\widehat{\gamma}_1), 
 \label{eq:nabla-general-2}
\end{align}
where $\widetilde{\gamma}_1^i\in \RN^{p-1}$ is formed from $\widetilde{\gamma}_1$ by removing its $i$th entry, see \eqref{eq:defn-Gamma}. By substituting (\ref{eq:nabla-general-1},\ref{eq:nabla-general-2}) back into \eqref{eq:st cnd in general updated}}, we conclude that as in the proof of Theorem \ref{det optimality} that $X_s\in \operatorname{St}(n,p)$ such that $X_s^* \Gamma X_s \in \operatorname{GL}(p)$ is a stationary point of Program~(\ref{eq:general program}) if and only if 
%which in particular yields that 
%\begin{equation}
%Z\in \GL(p) \Longrightarrow \nabla s_q (Z) \in \GL(p). 
%\label{eq:nabla s is nonsingular}
%\end{equation}
%\edit{In turn, }recalling that both $X^*\Gamma X\in\GL(p)$ and $X^* X\in \GL(p)$, we find that 
%\begin{equation}
%\nabla s_q(X_s^* X_s) \in \GL(p), \qquad \nabla s_q(X_s^* \Gamma X_s) \in \GL(p). 
%\end{equation}
%It then follows from \eqref{eq:st cnd in general updated} that $X_s\in \RN^{n\times p}_p$ with $X^*\Gamma X\in \GL(p)$ is a stationary point of Program \eqref{pca by det} if and only if  
\begin{equation}
\range( X_s)=\range(\Gamma X_s),
\label{eq:st cnd ranges new}
\end{equation} 
which is identical to \eqref{eq:st cnd ranges} in the proof of Theorem \ref{det optimality}, and consequently Lemmas \ref{lem:stationary} and \ref{lem:hessian calculated} therein apply here too. 
\edit{Moreover, note that 
\begin{align}
s_q(X_s^* X_s) & = s_q(I_p) \qquad (X_s \in \operatorname{St}(n,p) ) \nonumber\\
& = e_q(1_p) = {p \choose q}, \qquad \text{(see \eqref{eq:def of eq})}
\label{eq:used-in-nabla-1}
\end{align}
\begin{align}
s_q(X_s^* \Gamma X_s) & = s_q(\diag(\widetilde{\gamma}_1)) 
\qquad \text{(see Lemma \ref{lem:hessian calculated})}
\nonumber\\
& = e_q(\widetilde{\gamma}_1) . 
\qquad \text{(see \eqref{eq:def of eq})} 
\label{eq:used-in-nabla-2}
\end{align}
We can also revisit \eqref{eq:exp-for-lambda-s-thm2} to obtain that 
\begin{align}
\Lambda_s & = -\frac{1}{2} X_s^* \nabla f_{s_q}(X_s) 
\qquad \text{(see \eqref{eq:exp-for-lambda-s-thm2})}
\nonumber\\
& = - \frac{\diag(\widetilde{\gamma}_1) \diag(\widehat{\gamma}_1) }{{p \choose q}}
 + \frac{ e_q(\widetilde{\gamma}_1) {p-1 \choose q-1} I_p}{{p \choose q}^2} 
\qquad \text{(see (\ref{eq:nabla-general-1},\ref{eq:nabla-general-2},\ref{eq:used-in-nabla-1},\ref{eq:used-in-nabla-2}))} 
 \nonumber\\
 & = -\frac{{p-1 \choose q-1}}{{p \choose q }} \l(  \frac{\diag(\widetilde{\gamma}_1 \widehat{\gamma}_1 )}{{p-1 \choose q-1}} - \frac{e_q(\widetilde{\gamma}_1) I_p}{{p \choose q}}  \r). 
	 \label{eq:exp-for-lambda-final}
\end{align}}In particular, when $\phi = s_q$ \edit{(and consequently $\psi = e_q$)}, we next simplify the expression for Hessian in \eqref{eq:hessian lemma statement} by noting that 
\begin{equation*}
\phi(\I_p) = s_q(\I_p) = e_q(1_p) = {p \choose q}, 
\qquad \text{(see \eqref{eq:def of eq})}
\end{equation*}
\begin{equation*}
\partial_j \psi(\widetilde{\gamma}_1) = 
\partial_j e_q(\widetilde{\gamma}_1) = 
e_{q-1}(\widetilde{\gamma}_1^j),
\qquad j\in [p],
\end{equation*}
where $\widetilde{\gamma}_1^j\in\RN^{p-1}$ is formed from $\widetilde{\gamma}_1\in\RN^p$ by removing its $j$th entry, \edit{see \eqref{eq:def of gamma tilde}.} 
Moreover,
\begin{equation*}
\partial_1 \psi (1_p) = \partial_1 e_q(1_p) = e_{q-1}(1_{p-1}) = {p-1 \choose q-1},
\qquad \edit{\text{(see \eqref{eq:def of eq})}}
\end{equation*}
\begin{align}
f_{s_q}(X_s) 
& = 
 \frac{s_q(X_s^*\Gamma X_s)}{s_q(X_s^*X_s)} 
\qquad \mbox{(see \eqref{eq:def of g})} 
 \nonumber\\
& = \frac{s_q(\diag(\widetilde{\gamma}_1))}{s_q(\I_p)}
\qquad \l(\mbox{\eqref{eq:simp assumption lem statement} and } X_s\in\Stiefel(n,p) \r) \nonumber\\
& = \frac{e_q(\widetilde{\gamma}_1)}{e_q(1_p)} \nonumber\\
& =  \frac{e_q(\widetilde{\gamma}_1)}{{p \choose q}},
\qquad\edit{\text{(see \eqref{eq:def of eq})}}
\label{eq:value of g at Xs general} 
\end{align}
\begin{align}
\nabla^2 f_{s_q}(X_s)[\Delta,\Delta]
& =  \frac{1}{\phi(\I_p)} \sum_{i\in K^C}  \sum_{j=1}^p \l( \sigma_i^2 \partial_j e_q(\widetilde{\gamma}_1) - f_{s_q}(X_s) \partial_1 e_q(1_p) \r)  \Delta_{i,j}^2
\qquad \mbox{(see \eqref{eq:hessian lemma statement})}
\nonumber\\
& =  \frac{1}{{p\choose q}} \sum_{i\in K^C}  \sum_{j=1}^p \l( \sigma_i^2 e_{q-1}(\widetilde{\gamma}^j_1) - \frac{e_q(\widetilde{\gamma}_1) {p-1 \choose q-1}}{{p \choose q}} \r) \Delta_{i,j}^2 \nonumber\\
& = \frac{{p-1 \choose q-1}}{{p \choose q}} \sum_{i\in K^C}  \sum_{j=1}^p \l( \frac{\sigma_i^2 e_{q-1}(\widetilde{\gamma}^j_1)}{{p-1 \choose q-1}} - \frac{e_q(\widetilde{\gamma}_1) }{{p \choose q}} \r) \Delta_{i,j}^2.
%& = \frac{e_{q-1}(1_{p-1}) e_{q-1}(\widetilde{\gamma}^j_1) }{e_q(1_p)^2} \sum_{i\in K^C}  \sum_{j=1}^p \l( \frac{\sigma_i^2  e_q(1_p)}{e_{q-1}(1_{p-1})} - \frac{e_q(\widetilde{\gamma}_1) }{e_{q-1}(\widetilde{\gamma}^j_1)} \r) \Delta[i,j]^2. 
\label{eq:hessian lemma statement general}
\end{align}
\edit{In light of (\ref{eq:general program-rewritten},\ref{eq:st cnd in general}), let us record for the future reference that $\Delta\in \RN^{n\times p}$ is an ascent direction at  $X_s$ if 
\begin{align}
\nabla^2 f_{s_q}(X_s)[\Delta,\Delta]+ \langle  \Delta^*\Delta, \Lambda_s \rangle \ge 0,
\label{eq:ascent-cnd}
\end{align}}
and 
\begin{align}
    X_s^\top \Delta+\Delta^\top X_s = 0.
    \label{eq:skewSymm}
\end{align}
%
%We next study the sum in the last line above. To that end, the following  technical result is necessary, see Appendix \ref{sec:proof of technical lem} for the proof.
%\begin{lemma}\label{lem:technical lem}
%For a nonnegative vector $v\in \RN^p$, let $v_{j_M}$ be its largest entry and form $v^{j_M}\in\RN^{p-1}$ from $v\in\RN^p$ by removing $v_{j_M}$. Likewise, let $v_{j_m}$ be smallest entry of $v$ and form $v^{j_m}\in\RN^{p-1}$ by removing $v_{j_m}$.
%Then it holds that 
%\begin{equation}
%\frac{v_{j_m}  e_{q-1}(v^{j_m})}{{p-1 \choose q-1}} \le \frac{e_q(v)}{{p\choose q}} \le \frac{v_{j_M}  e_{q-1}(v^{j_M})}{{p-1 \choose q-1}},
%\qquad q\in  [p].
%\label{eq:technical ineq}
%\end{equation}
%\end{lemma}
By definition in \eqref{eq:def of gamma tilde},  $\{\sigma_{i}^2\}_{i\in J}$ are the distinct \edit{numbers appearing in} $\widetilde{\gamma}_1\in\RN^p$.
Suppose now that $\{\sigma_i\}_{i\in J}$ are \emph{not}  \edit{the unique numbers in the} $p$ leading singular values of $A$. Therefore there exist $i_0\notin K$ and $j_0\in [p]$  such that \edit{
\begin{align}
\sigma_{i_0}^2 >  \min_{i\in J} \sigma_i^2 = \min_{j\in[p]} \widetilde{\gamma}_{1,j} =: \widetilde{\gamma}_{1,j_0},
\qquad \edit{\text{(see \eqref{eq:def of gamma tilde})}}
\label{eq:not-p-leading-thm2}
\end{align}
and,\footnote{\edit{In \eqref{eq:not-p-leading-thm2}, the index $j_0$ might not be uniquely defined.}} moreover, 
\begin{align}
\frac{\sigma_{i_0}^2}{\widetilde{\gamma}_{1,j_0}} \ge \frac{\sigma_p^2}{\sigma_{p+1}^2}.
\label{eq:spec-gap-needed-thm2}
\end{align}}Let us set $\Delta\in\RN^{n\times p}$ such that $\Delta_{i_0,j_0}$ is its only nonzero entry and note that \eqref{eq:skewSymm} holds because $i_0\notin K$. For this choice of $\Delta$, \edit{we find that 
\begin{align}
& \nabla^2 f_{s_q}(X_s)[\Delta,\Delta]+\langle \Delta^* \Delta , \Lambda_s \rangle \nonumber\\
& = 
\frac{{p-1 \choose q-1}}{{p \choose q}} \l( \frac{\sigma_{i_0}^2 e_{q-1}(\widetilde{\gamma}_1^j)}{{p-1 \choose q-1}} - \frac{e_q(\widetilde{\gamma}_1) }{{p \choose q}}  \r) \Delta_{i_0,j_0}^2 + \Lambda_{s,j_0,j_0}^2 \Delta_{i_0,j_0}^2 
\qquad \text{(see \eqref{eq:hessian lemma statement general})}
\nonumber\\
& = \frac{{p-1 \choose q-1}}{{p \choose q}} \l( \frac{\sigma_{i_0}^2 e_{q-1}(\widetilde{\gamma}_1^{j_0})}{{p-1 \choose q-1}} - \frac{e_q(\widetilde{\gamma}_1)  }{{p \choose q}}  \r) \Delta_{i_0,j_0}^2 
 -\frac{{p-1 \choose q-1}}{{p \choose q }} \l(  \frac{\widetilde{\gamma}_{1,j_0} \widehat{\gamma}_{1,j_0} }{{p-1 \choose q-1}} - \frac{e_q(\widetilde{\gamma}_{1}) }{{p \choose q}}  \r) 
\Delta_{i_0,j_0}^2 
\qquad \text{(see \eqref{eq:exp-for-lambda-final})} \nonumber\\
& = \frac{ \sigma_{i_0}^2 e_{q-1}(\widetilde{\gamma}_1^{j_0})  - \widetilde{\gamma}_{1,j_0} \widehat{\gamma}_{1,j_0}  }{{p \choose q}} \Delta_{i_0,j_0}^2 \nonumber\\
& = \frac{(  \sigma_{i_0}^2   - \widetilde{\gamma}_{1,j_0}  ) e_{q-1}(\widetilde{\gamma}_1^j)  }{{p \choose q}} \Delta_{i_0,j_0}^2  
\qquad \text{(see \eqref{eq:nabla-general-2})} \nonumber\\
& >0.
\qquad \text{(see \eqref{eq:not-p-leading-thm2})}
\end{align}
That is, if $\{\sigma_i\}_{i\in J}$ are \emph{not} the unique members in the $p$ leading singular values of $A$, then there exists an ascent direction at $X_s$. 
%Moreover, if $\sigma_p >\sigma_{p+1}$, then it also holds that 
%\begin{align}
%\nabla^2 f_{s_q}(X_s)[\Delta,\Delta] + \langle \Delta^* \Delta, \Lambda_s \rangle  \ge 
%f_{s_q}(X_s) \l( \frac{\sigma_p^2}{\sigma_{p+1}^2} -1 \r) \| \Delta\|_F^2 . 
%\qquad \text{(similar to \eqref{eq:det-spec-gap-result})}
%\end{align}
The rest of the proof of Theorem \ref{general optimality} is now the same as that of Theorem \ref{det optimality}.
}

\section*{Acknowledgements}
RAH is supported by EPSRC grant EP/N510129/1. For this work, AE was supported by the Alan Turing Institute under the EPSRC grant EP/N510129/1 and also by the Turing Seed Funding grant SF019. AE would like to thank Stephen Becker and David Bortz for pointing out the connection to D-optimality in optimal design, and Mike Davis for the connection to independent component analysis. 

\appendix

\section{Proof of Lemma \ref{lem:stationary} \label{sec:proof of lemma stationary}}

Consider $X_{s}\in\Stiefel(n,p)$ such that $X_s^*\Gamma X_s \in \GL(p)$ and 
\begin{equation}
\range(X_s) = \range(\Gamma X_s).
\label{eq:st range repeated}
\end{equation}
Each row of $X_s$ naturally corresponds to a singular value of $A$, namely the $i$-th row corresponds to 
$\sigma_i$, where $\sigma_i^2$ is the $i$-th diagonal entry of $\Gamma$. 
Let $J\subset[n]$ be the index set such that $\Sigma_J :=\{\sigma_i\}_{i\in J}$ is  the set of \emph{distinct} singular values corresponding to the nonzero rows of $X_s$.\footnote{\edit{Throughout, we treat $\{\sigma_i\}_{i\in J}$ and similar items as sequences (rather than sets) to allow for repetitions. }}
For future reference, let us record  that $\Sigma_J$ contains only positive singular values, namely 
\begin{equation}
\Sigma_J \subset \RN_+. 
\label{eq:positive}
\end{equation}  
Indeed,  if $0\in \Sigma_J$, namely if $\sigma_n = 0$, then the rows of $X_s$ corresponding to $\sigma_n$ are zero too thanks to  \eqref{eq:st range repeated} and consequently $0\notin \Sigma_J$, which leads to a contradiction.  
Let us now set  
\begin{equation}
\text{dim}\left(\sigma_{i}\right):=\max\left[p-\sum_{j\in J,\,j\ne i}n_{j},1\right],\qquad i\in J,\label{eq:dim}
\end{equation}
for short, where $n_i$ is the multiplicity of $\sigma_i$. 
Fix $i_{0}\in J$. Consider $\mathcal{K}_{i_{0}}$, the collection
of all index sets $K\subset[n]$ of size $p$ such that 
\begin{equation}
\sigma_{i_{0}}\in\mbox{unique}\l( \{\sigma_{i}\}_{i\in K} \r)\subseteq\Sigma_{J},
\end{equation}
where $\mbox{unique}(\{\sigma_{i}\}_{i\in K})$ returns the distinct members of the set $\{\sigma_{i}\}_{i\in K}$. In words, every index set $K\in \mathcal{K}_{i_0}$ contains $\sigma_{i_0}$ and $p-1$ other (not necessarily distinct) singular values of $A$ corresponding to nonzero rows of $X_s$. Consider
an arbitrary $K\in\mathcal{K}_{i_{0}}$. It follows from  (\ref{eq:dim}) that  
\begin{equation}
\{\sigma_{i}\}_{i\in K}  \text{ contains at least } \text{dim}(\sigma_{i_0}) \text{ copies of } \sigma_{i_{0}}.
\label{eq:no of copies in I}
\end{equation}
On the other hand, \eqref{eq:st range repeated} implies that there exists $B\in\mathbb{R}^{p\times p}$ such that 
\begin{equation}
X_s  B = \Gamma X_s. 
\label{eq:st cnd repeated}
\end{equation} 
By multiplying both sides above by $X_s^*$ and using the fact that $X_s\in\Stiefel(n,p)$, we infer from \eqref{eq:st cnd repeated} that 
\begin{equation}
B=X_s^*\Gamma X_s \in \GL(p),
\label{eq:B full rank}
\end{equation}
where the invertibility of $B$ follows from the assumption of Lemma \ref{lem:stationary}.  In addition, \eqref{eq:st cnd repeated} means that each row of $X_s$ is an eigenvector of $B$. 
By restricting (\ref{eq:st cnd repeated}) to the index set $K$, we
find that 
\begin{equation}
 X_{s}[K,:] \cdot B =\Gamma[K,K]\cdot X_{s}[K,:]
 \text{ is the eigen-decomposition of } B,
 \label{eq:rest is eig decomp}
\end{equation}
%which is the eigen-decomposition of $B$ 
because $K\in \mathcal{K}_{i_0}$ is a set of size $p$ by the definition of $\mathcal{K}_{i_0}$ earlier. Above, we used MATLAB's matrix notation. For example, $X_s[K,:]\in \RN^{p\times p}$ above is the row-submatrix  of $X_s$ corresponding to the rows indexed by $K$. 
It follows from \eqref{eq:rest is eig decomp} that 
\begin{equation}
\sigma_{i_{0}} \text{ is an eigenvalue of } B \text{ with the multiplicity
of at least } \text{dim}(\sigma_{i_{0}}),
\end{equation}
because, by \eqref{eq:no of copies in I}, $\{\sigma_i\}_{i\in K}$ contains at least $\mbox{dim}(\sigma_{i_0})$ copies of $\sigma_{i_0}$. In fact, 
 there exists an index set $K_{0}\in\mathcal{K}_{i_{0}}$  such that  $\{\sigma_{i}\}_{i\in K_{0}}$  contains  {\emph{exactly }} $\text{dim}(\sigma_{i_{0}})$ { copies of } $\sigma_{i_{0}}$.\footnote{Indeed, if $\text{dim}(\sigma_{i_0})>1$, such an index set $K_0$ would include $n_{i}$ copies of singular value $\sigma_i$, for every  $\sigma_i \ne \sigma_{i_0}$ with $i\in K$. The construction is similar if $\text{dim}(\sigma_{i_0})=1$. 
}
It follows that 
\begin{equation}
\sigma_{i_{0}} \text{ is an eigenvalue of } B \text{ with the multiplicity
of exactly } \text{dim}(\sigma_{i_{0}}).
\label{eq:eval exact mult}
\end{equation}
By \eqref{eq:B full rank}, $B$ is full-rank and it follows from \eqref{eq:eval exact mult} that the corresponding eigenvectors of $B$ span a $\text{dim}(\sigma_{i_{0}})$-dimensional subspace of $\mathbb{R}^{p}$,  namely the geometric multiplicity of $\sigma_{i_0}$ is $\mbox{dim}(\sigma_{i_0})$.  
Since every row of $X_{s}$ is an eigenvectors of $B$ by \eqref{eq:rest is eig decomp}, it follows that the 
\begin{equation}
\text{rows of } X_{s} \text{ that correspond to }\sigma_{i_{0}}
 \text{ span a }\text{dim}(\sigma_{i_{0}})\text{-dimensional subspace of }\RN^p. 
\end{equation}
Since the choice of ${i_0}\in J$ was arbitrary above, we find for every $i\in J$ that 
\begin{align}
\text{the rows of }X_{s} \text{ corresponding to }\sigma_{i} \text{ span a }\text{dim}(\sigma_{i})\text{-dimensional}\nonumber\\
\text{subspace of }\mathbb{R}^{p} \text{, denoted by }S_i\in \Grassman(p,\mbox{dim}(\sigma_i)).
\label{eq:def of Si}
\end{align}
Because $B$ is symmetric by its definition in \eqref{eq:B full rank}, these subspaces are  orthogonal  to one another, namely 
\begin{equation}
S_i \perp S_{j}, \qquad i\ne j \mbox{ and } i,j\in J. 
\label{eq:orth of Sj} 
\end{equation}
On the other hand, note that 
\begin{align*}
B & =X_{s}^{*}\Sigma X_{s}
\qquad \mbox{(see \eqref{eq:st cnd repeated})}
\\
 & =\sum_{i\in J}\sigma_{i}\left(\sum_{\sigma_{j}=\sigma_{i}}X_{s}\left[j,:\right]^{*}X_{s}[j,:]\right)+\sum_{\sigma_{j}\notin\Sigma_{J}}\sigma_{j}\cdot X_{s}[j,:]^{*}X_{s}[j,:]\\
 & =:\sum_{i\in J}\sigma_{i}Y_{i}+\sum_{\sigma_{j}\notin\Sigma_{J}}\sigma_{j}\cdot X_{s}[j,:]^{*}X_{s}[j,:]\\
 & =\sum_{i\in J}\sigma_{i}Y_{i},
\end{align*}
where the last line above follows from the definition of $J$, namely any singular value $\sigma_i\notin \Sigma_J$ corresponds to a zero row of $X_s$. For every $i\in J$, note that 
\begin{equation}
\range(Y_{i}) =  S_i
\label{eq:range of Yi}
\end{equation}
by definition of $S_i$ in \eqref{eq:def of Si}. It therefore follows from \eqref{eq:orth of Sj} that $\{Y_{i}\}_{i\in J}$
are pairwise orthogonal matrices, namely 
\begin{equation}
Y_{i}^{*}Y_{j}=0,\qquad i\ne j \mbox{ and } i,j\in J.\label{eq:obs}
\end{equation}
Therefore,  $B=\sum_{i\in J}\sigma_{i}Y_{i}$ is the eigen-decomposition
of $B$ and, because $B$ is full-rank by \eqref{eq:B full rank}, we find that 
\begin{align}
p
& = \sum_{i\in J,\, \sigma_i \ne 0} \mbox{dim}(\text{span}(Y_i)) \nonumber\\
& = \sum_{i\in J,\, \sigma_i \ne 0} \mbox{dim}(S_i) \qquad \text{(see \eqref{eq:range of Yi})}\nonumber\\
& =\sum_{i\in J,\,\sigma_{i}\ne0}\text{dim}(\sigma_{i}) 
\qquad \text{(see \eqref{eq:def of Si})}
\nonumber\\
& = \sum_{i\in J} \mbox{dim}(\sigma_i).
\qquad \text{(see \eqref{eq:positive})}
\label{eq:sum to p}
\end{align}
This completes
the proof of Lemma \ref{lem:stationary}. 

\section{Proof of Lemma \ref{lem:hessian calculated} \label{sec:proof of stability lemma}}

Suppose that $X_{s}\in\Stiefel(n,p)$ satisfies 
\begin{equation}
\range(X_s)=\range(\Gamma X_s),
\label{eq:st cnd repeated again}
\end{equation}
which implies that 
\begin{equation}
X_s B = \Gamma X_s, 
\qquad \mbox{where } B = X^*_s \Gamma X_s. 
\label{eq:inv subspace}
\end{equation}
 Let $J\subseteq[n]$ be the corresponding index set prescribed in Lemma \ref{lem:stationary}, and recall that $\{\sigma_i\}_{i\in J}$ are distinct by Item 1 in Lemma \ref{lem:stationary}. Consider an index set $K\supseteq J$ such that $\{\sigma_{i}\}_{i\in K}$
contains all available copies of the singular values listed in $\{\sigma_{i}\}_{i\in J}$. Let $k $ denote the size of $K$. 
For convenience, we define
\[
X_{s,1}:=X_{s}\left[K,:\right]\in\mathbb{R}^{k \times p},\qquad X_{s,2}:=X_{s}\left[K^C,:\right]\in\mathbb{R}^{(n-k)\times p},
\]
\begin{equation}
\Gamma_{1}:=\Gamma\left[K,K\right]\in\mathbb{R}^{k\times k},\qquad\Gamma_{2}:=\Gamma\left[K^C,K^C\right]\in\mathbb{R}^{(n-k)\times(n-k)},
\label{eq:def of Gammas}
\end{equation}
where we used MATLAB's matrix notation above. For example, $X_s[K,:]$ is the restriction of $X_s$ to the rows indexed in $K$. Also, $K^C$ is the complement of index set $K$ with respect to $[n]$. In particular, Item 2 in Lemma~\ref{lem:stationary} immediately implies that 
 \begin{equation}
X_{s,2}=0, 
\label{eq:Xs2 is zero}
\end{equation}
and, consequently,
\begin{align}
X_{s,1}^{*}X_{s,1}& = X_{s,1}^* X_{s,1}+X_{s,2}^* X_{s,2} \nonumber\\
& =X_{s}^{*}X_{s} \nonumber\\
& =\I_{p}.
\qquad \l( X_s\in \Stiefel(n,p)\r)
\end{align}
That is,
\begin{equation}
X_{s,1}\in\Stiefel(k,p).
\label{eq:X1 is Stiefel}
\end{equation} 
Note that \eqref{eq:st cnd repeated again} holds also after a change of basis from $X_s$ to $X_s \Theta$ for invertible $\Theta\in\GL(p)$. 
Therefore, thanks to \eqref{eq:st cnd repeated again} and Item 4 in Lemma \ref{lem:stationary}, we can assume without loss of generality that
the supports of rows and also columns of $X_{s,1}$  are disjoint. More specifically, with the enumeration $J=\{i_1,i_2,\cdots\}$, we assume without loss of generality that 
\begin{equation}
X_{s,1}  = \l[
\begin{array}{ccc}
X_{s,1,1} & 0 & \cdots \\
0 & X_{s,1,2} & \cdots \\
\vdots & \vdots & \ddots 
\end{array}
\r] \in\RN^{k\times p},
\label{eq:blocks of Xs1}
\end{equation}
where the rows of \edit{the block} $X_{s,1,1}$ corresponds to the singular value $\sigma_{i_1}$ and has $\dim(\sigma_{i_1})$ columns, \edit{the block} $X_{s,1,2}$ corresponds to $\sigma_{i_2}$ and so on. In particular, \eqref{eq:X1 is Stiefel} implies that 
\begin{equation}
X_{s,1,1}^* X_{s,1,1} = \I_{\dim(\sigma_{i_1})},
\qquad 
X_{s,1,2}^* X_{s,1,2} = \I_{\dim(\sigma_{i_2})},
\qquad \cdots
\label{eq:each block orth}
\end{equation} 
namely,  $X_{s,1,1}$ has orthonormal columns,  so do $X_{s,1,2}$ and the rest of \edit{the diagonal} blocks of $X_{s,1}$. Another necessary ingredient in our analysis below is the observation that 
\begin{align}
X_{s}^{*}\Gamma X_{s}& = X_{s,1}^{*}\Gamma_{1}X_{s,1}+X_{s,2}^{*}\Gamma_{2}X_{s,2} \nonumber\\
& = X_{s,1}^{*}\Gamma_{1}X_{s,1}
\qquad \mbox{(see \eqref{eq:Xs2 is zero})} \nonumber\\
& = \l[ 
\begin{array}{ccc}
\sigma_{i_1}^2 \cdot X_{s,1,1}^* X_{s,1,1} & 0 & \cdots\\
0 & \sigma_{i_2}^2 \cdot X_{s,1,2}^* X_{s,1,2} & \cdots \\
\vdots & \vdots & \ddots 
\end{array}
 \r] 
  \nonumber\\
& = \l[ 
\begin{array}{ccc}
\sigma_{i_1}^2 \cdot \I_{\dim(\sigma_{i_1})} & 0 & \cdots\\
0 & \sigma_{i_2}^2 \cdot \I_{\dim(\sigma_{i_1})} & \cdots \\
\vdots & \vdots & \ddots 
\end{array}
\r]
\qquad \text{(see \eqref{eq:each block orth})}
 \nonumber\\
& =:\widetilde{\Gamma}_{1}\in\mathbb{R}^{p\times p} \nonumber\\
& =: \diag(\widetilde{\gamma}_1),
\label{eq:needed later}
\end{align}
namely, the diagonal matrix $\widetilde{\Gamma}_{1}$ contains $\dim(\sigma_{i_1})$ copies of $\sigma_{i_1}^2$, $\dim(\sigma_{i_2})$ copies of $\sigma_{i_2}^2$, and so on. 
To compute the Hessian of $f_\phi$, we make a small perturbation to its argument. 
To be specific, consider $\Delta\in \RN^{n\times p}$ that is supported only on the rows indexed by $K^C$ and let 
\begin{equation}
%\Delta_{1}:=\Delta\left[K,:\right]\in\mathbb{R}^{k\times p},\qquad
\Delta_{2}:=\Delta\left[K^C,:\right]\in\mathbb{R}^{(n-k)\times p}
\label{eq:def of Delta2}
\end{equation}
be the nonzero block of $\Delta$. Note in particular that 
\begin{equation}
X_s^* \Delta = 0,
\label{eq:orth of X n Delta}
\end{equation}
because by construction  $X_s$ and $\Delta$ are supported on the rows indexed by $K$ and $K^C$, respectively, see \eqref{eq:Xs2 is zero}. 
Let $h_\Gamma(X) = \phi(X^*\Gamma X)$ for short and note that 
\begin{align}
h_\Gamma(X_s+\Delta ) & = \phi((X_s+\Delta)^*\Gamma ( X_s+\Delta) ) \nonumber\\
& = \phi( X_s^* \Gamma X_s + X_s^*\Gamma \Delta + \Delta^* \Gamma X_s + \Delta^* \Gamma \Delta)  \nonumber\\
& = \phi( X_s^* \Gamma X_s + B X_s^* \Delta + \Delta^* X_s B + \Delta^* \Gamma \Delta ) 
\qquad \mbox{(see \eqref{eq:inv subspace})}
\nonumber\\
& = \phi (X_s^* \Gamma X_s + \Delta^* \Gamma \Delta)
\qquad \mbox{(see \eqref{eq:orth of X n Delta})}  \nonumber\\
& = \phi ( \widetilde{\Gamma}_1+\Delta_2^*\Gamma_2 \Delta_2 )
\qquad \mbox{(see (\ref{eq:needed later},\ref{eq:def of Delta2},\ref{eq:def of Gammas}))} \nonumber\\
& = \phi(\widetilde{\Gamma}_1) + \langle \nabla \phi(\widetilde{\Gamma}_1),\Delta_2^* \Gamma_2 \Delta_2\rangle + O(\|\Delta_2\|^3)
\qquad \mbox{(Taylor expansion)}
 \nonumber\\
& = \phi(X_s^*\Gamma X_s) + \langle \nabla \phi(\widetilde{\Gamma}_1),\Delta_2^* \Gamma_2 \Delta_2\rangle + O(\|\Delta\|^3),
\qquad \mbox{(see (\ref{eq:needed later},\ref{eq:def of Delta2}))}
\end{align}
where we used the standard Big-$O$ notation above. 
Recall from \eqref{eq:needed later} that $\widetilde{\Gamma}_1 = \diag(\widetilde{\gamma}_1)$. Because $\phi$ is by assumption a spectral function with the corresponding symmetric function $\psi$,  \eqref{eq:grad of spec} implies that 
\begin{equation}
\nabla \phi(\widetilde{\Gamma}_1) = \diag(\nabla \psi(\widetilde{\gamma}_1)) = \diag\l( 
\l[
\begin{array}{ccc}
\partial_1 \psi(\widetilde{\gamma}_{1,1}) & \cdots \partial_p \psi(\widetilde{\gamma}_{1,p})
\end{array}
\r]^*
 \r),
\end{equation}
which allows us to rewrite the last line above as 
\begin{align}
h_\Gamma(X_s+\Delta)  & = \phi(X_s^*\Gamma X_s) +\langle \diag (\nabla \psi( \widetilde{\gamma}_1)),\Delta_2^* \Gamma_2 \Delta_2\rangle +O(\|\Delta\|^3) \nonumber\\
& =  \phi(X_s^*\Gamma X_s) +
 \sum_{i\in K^C}  \sum_{j=1}^p \sigma_i^2 \cdot \partial_j \psi(\widetilde{\gamma}_1) \cdot \Delta[i,j]^2 + O(\|\Delta\|^3),
 \label{eq:expansion of h gamma}
\end{align}
where $\Delta[i,j]$ is the $[i,j]$th entry of $\Delta$. 
Let $1_p\in\RN^p$ be the vector of all ones. After setting $h_{\I}(X) = \phi(X^*X)$ and after replacing $\Gamma$ with $\I_n$ above, we find that 
\begin{align}
h_{\I}(X_s+\Delta) & = \phi(X_s^* X_s) +
 \sum_{i\in K^C}  \sum_{j=1}^p \partial_j \psi(1_p) \cdot \Delta[i,j]^2 + O(\|\Delta\|^3) \nonumber\\
 & = \phi(\I_p) +
 \partial_1 \psi(1_p)  \sum_{i\in K^C}  \sum_{j=1}^p  \Delta[i,j]^2 + O(\|\Delta\|^3), 
 \label{eq:expansion of h I}
\end{align}
where in the last line above we used the fact thta $X_s\in \Stiefel(n,p)$ and that  $\psi$ is a symmetric function, hence $\partial_j \psi(1_p) = \partial_1 \psi(1_p)$ for every $j\in [p]$. Since $f_\phi = h_\Gamma/h_{\I}$ by definition, (\ref{eq:expansion of h gamma},\ref{eq:expansion of h I}) imply that 
\begin{align}
f_\phi(X_s+\Delta) & = \frac{h_\Gamma(X_s+\Delta)}{h_{\I}(X_s+\Delta)} \nonumber\\
& = \frac{\phi(X_s^*\Gamma X_s) +
 \sum_{i\in K^C}  \sum_{j=1}^p \sigma_i^2 \cdot \partial_j \psi(\widetilde{\gamma}_1) \cdot \Delta[i,j]^2 + O(\|\Delta\|^3)}{\phi(\I_p) +
\partial_1 \psi(1_p) \sum_{i\in K^C}  \sum_{j=1}^p \Delta_{i,j}^2 + O(\|\Delta\|^3)} 
 \qquad \mbox{(see (\ref{eq:expansion of h gamma},\ref{eq:expansion of h I}))} \nonumber\\
 & = \l( \phi(X_s^*\Gamma X_s) +
 \sum_{i\in K^C}  \sum_{j=1}^p \sigma_i^2 \cdot \partial_j \psi(\widetilde{\gamma}_1) \cdot \Delta[i,j]^2 + O(\|\Delta\|^3) \r) \nonumber\\
 & \qquad \cdot \frac{1}{\phi(\I_p)}  \l(1 - \frac{\partial_1 \psi(1_p)}{\phi(\I_p)}  \sum_{i\in K^C}  \sum_{j=1}^p  \Delta[i,j]^2 + O(\|\Delta\|^3  \r) 
\qquad \l( \frac{1}{1+a}=1-a+O(a^2) \r) 
 \nonumber\\
 & =  \frac{\phi(X_s^*\Gamma X_s)}{\phi(\I_p)} - \frac{\phi(X_s^* \Gamma X_s)}{\phi(\I_p)} \cdot \frac{\partial_1 \psi(1_p)}{\phi(\I_p)} \sum_{i\in K^C} \sum_{j=1}^p \Delta[i,j]^2 \nonumber\\
 & \qquad  + \frac{1}{\phi(\I_p)} \sum_{i\in K^C}  \sum_{j=1}^p \sigma_i^2 \cdot \partial_j \psi(\widetilde{\gamma}_1) \cdot \Delta[i,j]^2  + O(\|\Delta\|^3) \nonumber\\ 
 & = f_\phi(X_s) - f_\phi(X_s) \frac{\partial_1 \psi(1_p)}{\phi(\I_p)} \sum_{i\in K^C} \sum_{j=1}^p \Delta[i,j]^2  
 \nonumber\\
& \qquad + \frac{1}{\phi(\I_p)} \sum_{i\in K^C}  \sum_{j=1}^p \sigma_i^2 \cdot \partial_j \psi(\widetilde{\gamma}_1) \cdot \Delta[i,j]^2+ O(\|\Delta\|^3),
\qquad \text{(see \eqref{eq:def of g})}
\end{align}
and, consequently,
\begin{equation}
\nabla^2 f_\phi(X_s)[\Delta,\Delta] =  \frac{1}{\phi(\I_p)} \sum_{i\in K^C}  \sum_{j=1}^p \l( \sigma_i^2 \partial_j \psi(\widetilde{\gamma}_1) - f_\phi(X_s) \partial_1 \psi(1_p) \r)  \Delta_{i,j}^2,
\end{equation}
for our particular choice of $\Delta$ that satisfies $\Delta[K,:]=0$. Here, the bilinear operator $\nabla^2 f_\phi(X_s) : \RN^{n\times p} \times \RN^{n\times p} \rightarrow \RN$ is the Hessian of $f_\phi$ at $X_s$. This completes the proof of Lemma \ref{lem:hessian calculated}.

\bibliographystyle{unsrt}
\bibliography{det}

\end{document}